\documentclass[11pt]{article}
\usepackage{amsmath,amsthm,amssymb,latexsym,amsfonts,epsfig, color, geometry, enumitem, hyperref,float,multirow,bm}
\usepackage{tikz}
\usepackage[numbers,sort]{natbib}
 \bibpunct[, ]{[}{]}{,}{n}{}{,}%

\newcommand{\ip}[2]{\langle #1 , #2 \rangle}    

\newcommand{\st}{\operatorname{s.t.}}

\newcommand{\matr}[1]{\begin{bmatrix} #1 \end{bmatrix}}    

\newcommand{\Sym}{\mathbb S}

\newcommand\R{{\mathbb R}}

\newcommand{\tr}{^{\top}}

\newcommand{\eps}{\varepsilon}

\newtheorem{theorem}{Theorem}
\newtheorem{lemma}{Lemma}
\newtheorem{proposition}{Proposition}
\newtheorem{corollary}{Corollary}

\newtheorem{remark}{Remark}
\newtheorem{conjecture}{Conjecture}

\newcommand{\lrc}[1]{\left \{ {#1} \right \}}

\definecolor{mygreen}{cmyk}{0.82,0.11,1,0.25}

\newcommand{\vecc}{\operatorname{vec}}
\title{The maximum $k$-colorable subgraph problem and related problems}

\author{
Olga Kuryatnikova\thanks{Erasmus University Rotterdam,  The  Netherlands, \href{mailto:kuryatnikova@ese.eur.nl}{kuryatnikova@ese.eur.nl}}
\and
Renata Sotirov\thanks{Department of Econometrics and OR, Tilburg University,  The  Netherlands, \href{mailto:r.sotirov@uvt.nl}{r.sotirov@uvt.nl}}
\and
Juan C. Vera\thanks{Department of Econometrics and OR, Tilburg University,  The  Netherlands,\href{mailto:j.c.veralizcano@uvt.nl}{j.c.veralizcano@uvt.nl}}
}

\date{}

\begin{document}
\maketitle

\begin{abstract}
The maximum $k$-colorable subgraph (M$k$CS)  problem   is to find an induced $k$-colorable subgraph with maximum cardinality in a given graph.
This paper is an in-depth  analysis of the M$k$CS problem that considers  various semidefinite programming relaxations
including their theoretical and numerical comparisons.
To simplify these relaxations we exploit the symmetry arising from permuting the colors,  as well as the  symmetry of the given graphs when applicable.
We also show how to exploit invariance under permutations of the subsets for other partition problems and
how to use  the M$k$CS problem to derive bounds on the chromatic number of a graph.

 Our numerical results verify that the proposed relaxations  provide strong  bounds for  the M$k$CS problem,
 and that those outperform existing bounds for most of the test instances. \\

\end{abstract}


\section{Introduction} \label{sec:introMkCS}

The maximum $k$-colorable subgraph (M$k$CS) problem
 is to find  the largest induced subgraph in a given graph that can be colored in $k$ colors such that no two adjacent vertices have the same color.
The M$k$CS  problem is also known as the maximum $k$-partite induced subgraph problem since the $k$-coloring corresponds to a $k$-partition of the subgraph.
The M$k$CS problem for $k=2$ is also known as the maximum bipartite subgraph problem.

In the literature, the name ``maximum $k$-colorable subgraph problem" is sometimes used for the maximum $k$-cut problem~\cite{kCut1,kCut4}. In the latter problem one searches for the partition of the graph into $k$ subsets such that the number of edges crossing the subsets is maximized. If one colors  vertices in the resulting subsets in different colors, the crossing edges are properly colored, that is, the endpoints of these edges have different colors.
However,  the M$k$CS  and  the maximum $k$-cut  are different  problems, and we do not consider the latter problem in this paper.
We refer interested readers to~\cite{Rendl2012,kCut1,kCut2,kCut3,maxCut,deKlerk2004} for more information on the maximum $k$-cut problem
and related semidefinite programming (SDP) relaxations.

The M$k$CS problem falls into the class of  NP-hard problems considered by~\citet{ksubgrNP}. Moreover, even approximating this problem is NP-hard  \cite{ksubgrApprox}. For $k=1$ the M$k$CS problem reduces to the famous maximum stable set problem, which was shown to be NP-hard by~\citet{Karp1972}.
Another well-known problem from the list of~\citet{Karp1972} related to the M$k$CS problem is the chromatic number problem.
The chromatic number problem is to determine whether the vertices of a given graph can be colored in $k$ colors.
If one can solve the M$k$CS problem for any given number of colors, then one  can also solve the maximum stable set and the chromatic number problems. However, efficient algorithms for the latter two problems do not necessarily result in efficient algorithms for the M$k$CS  problem.
For instance,  the chromatic number and the stability number on perfect graphs can be computed in polynomial time while the M$k$CS problem is NP-hard on chordal graphs which is a subfamily of the set of perfect graphs~\cite{chordal1}.
However, there are special classes of graphs for which the  M$k$CS problem is polynomial-time solvable.
Some examples are graphs where every odd cycle has two non-crossing chords for any $k$~\cite{iTriang}, clique-separable graphs for $k=2$~\cite{iTriang}, chordal graphs for fixed $k$~\cite{chordal1}, interval graphs for any $k$~\cite{chordal1}, circular-arc graphs and tolerance graphs for $k=2$~\cite{colorThesis}.
It is  known that  the size of the maximum stable set of the  Kneser graphs $K(v,d)$, where $v\geq 2d$, equals $\binom{v-1}{d-1}$, see \cite{ErdoRado}.
For the Kneser graphs $K(v,2)$, the size of the maximum $k$-colorable subgraph is also known, see e.g., \cite{Furedi}.

The M$k$CS problem for $k=2$ has been studied, among others, by~\citet{bip1,bip2,bip3,Godsil1,BresarPabon}.
However, the M$k$CS problem  for  $k>2$ is rarely considered in the literature.
\citet{JanuschowskiBCP} notice that such a lack of attention might be  related to the connection of the M$k$CS problem to the earlier mentioned prominent problems.
One of the few significant sources of information of the M$k$CS problem for $k> 2$, but also $k=2$, are \citet{colorThesis} and \citet{Narasimhan1}, where the authors introduce
an upper bound on the optimal value of the M$k$CS problem called the generalized $\vartheta$-number.
 Namely, they generalize the concept of the  famous $\vartheta$-number by~\citet{Lovasz},     which is an upper bound on the size of the maximum stable set  of a graph.
\citet{Alizadeh} formulated the generalized $\vartheta$-number problem using SDP.
\citet{Mohar} present the bound by~\citet{Narasimhan1} among important applications of eigenvalues of graphs in combinatorial optimization.
To our knowledge, the quality of the generalized $\vartheta$-number has never been evaluated.

Other related work considers integer programming (IP) formulations of the M$k$CS  problem, see e.g.~\cite{JanuschowskiBCP,benchmark1,JanuschowskiOrb}.
\citet{JanuschowskiBCP}, and~\citet{benchmark1}  provide computational results for $k\geq 2$.
In particular,  \citet{JanuschowskiBCP} consider graphs with symmetry, which enables them to provide numerical results for large  graphs up to $1085$ vertices.
\citet{online} analyze the performance of various existing online algorithms for the M$k$CS problem,  \citet{BresarPabon} provide theoretical lower and upper bounds for
the M$k$CS problem for $k\geq 2$ on the Kneser graphs.

It is worth mentioning that the M$k$CS problem  has a number of applications, such as channel assignment in spectrum sharing networks
(e.g., Wi-Fi or cellular)~\cite{networksApp1,networksApp2,networksApp3,KosterScheffel,scheduling}, VLSI design~\cite{VLSI1,VLSI2} and human genetic research~\cite{biology,VLSI2}.
Let us provide a brief description of one of the applications, i.e., the wavelength assignment problem in optical networks \cite{KosterScheffel}.
The problem is to assign wavelengths to as many light paths as possible considering that intersecting light paths do not use the same wavelength.
Hence, light paths correspond to vertices and wavelengths to colors.
There is an edge between two vertices if the corresponding light paths intersect.
The number of wavelengths is restricted by the capacity of the network.\\

\noindent \textbf{Outline and main results}  \medskip

This paper is an in-depth  analysis of  the M$k$CS problem that includes   various semidefinite programming relaxations
and their theoretical and numerical comparisons.
This analysis also extends results for the M$k$CS problem to  other graph partition problems,
but also relates the M$k$CS results with the stable set  and the chromatic number problems.

We begin our study with the eigenvalue bound by~\citet{Narasimhan1}, that is also known as the generalized $\vartheta$-number.
We present the generalized $\vartheta$-number as the optimal solution of  an SDP relaxation, see also \citet{Alizadeh}.
In order to strengthen the  mentioned SDP  relaxation we  add non-negativity constraints to the matrix variable, and
call the solution of the resulting SDP relaxation the generalized $\vartheta'$-number.
This number can be seen as the generalization of the Schrijver $\vartheta'$-number~\cite{SchrijverTheta}.
 Both, the generalized $\vartheta$-number and $\vartheta'$-number, require  solving an SDP relaxation with one matrix variable of the order $n$, where $n$ is the number of vertices in the graph.

Next, we derive vector lifting and matrix lifting  based SDP relaxations.
The sizes of matrix variables in the resulting relaxations depend on $n$ and $k$.
We  reduce the sizes of the SDP relaxations by exploiting the invariance of the M$k$CS problem under permutations of the colors.
In particular,  we exploit the fact that  all constraints in the relaxations are satisfied for any color labeling, and the objective does not change if the labeling changes.
This property is inherited by our SDP relaxations from  the IP formulations of the M$k$CS problem.
By exploiting color invariance, our matrix lifting SDP relaxation reduces to a model with one SDP constraint of  the order $n{+}1$,
and our strongest SDP relaxation reduces to a model with two SDP constraints of the order $n{+}1$ and $n$, respectively, for a graph with $n$ vertices.
Thus, it turns out that matrix sizes in the vector and matrix lifting relaxations are independent of $k$.

To  strengthen our relaxations, we add inequalities from the boolean quadric polytope.
We also show how to further reduce our SDP relaxations for highly symmetric graphs. This  reduction results in  a linear  program  arising from
the generalized $\vartheta'$-number, or in programs with a linear objective, one second order cone constraint, and many linear constraints.

Since the $k$-colorable subgraph problem is also a graph partition problem,
we are able to apply our symmetry reduction approach based on the invariance under permutations of the subsets to other partition problems.
In particular, we  prove in an elegant way that the vector and matrix lifting relaxations for the $k$-equipartition problem  are equivalent.
We obtain a similar result for the max-$k$-cut problem.

Finally, we evaluate the quality of all here presented SDP upper bounds on instances from \cite{JanuschowskiBCP,benchmark1}.
We also propose two heuristic approaches to compute lower bounds for the M$k$CS problem.
Our computational results show that our lower and upper bounds for the M$k$CS problem
 are strong and can be  computed efficiently for dense graphs or highly symmetric graphs.
 Our lower bounds on the chromatic number of a graph are competitive  with existing bounds in the literature.  \\

This paper is organized as follows.
We present several equivalent integer programming formulations of  the M$k$CS problem in Section \ref{sec:form}.
In Section \ref{sec:theta} we present first the generalized $\vartheta$-number by~\citet{Narasimhan1}, and then
the related strengthened bound that we call the generalized $\vartheta'$-number.
In Section \ref{sec:ipv} we first propose two vector lifting SDP relaxations and compare them, and then we show how to
apply symmetry reduction on colors in order to reduce their sizes.
In order to further tighten our SDP relaxations we consider adding inequalities from the boolean quadric polytope in Section \ref{subsec:bqp}.
In Section \ref{sec:alpha} we prove that  the Schrijver's $\vartheta'$-number on the Cartesian product of the complete graph on $k$ vertices and the graph under the consideration equals the optimal value of our weaker vector lifting relaxation for the original graph.
A matrix lifting SDP relaxation and its symmetry reduced version are given in Section~\ref{sec:ipm}.
In Section  \ref{sec:GraphSym} we show how to further simplify and reduce our SDP relaxations by
exploiting symmetry  of graphs. The application of symmetry reduction on colors is extended to other graph partition problems in Section \ref{subsec:partit}.
Section~\ref{sec:num} contains numerical results.  We summarize the results in Section~\ref{sec:conclusion}.

\medskip
\medskip

\noindent \textbf{Notation.}
The space  of  $n{\times}n$  symmetric matrices is denoted by $\Sym^n$.
For $A,B \in \Sym^n$, the trace inner product of $A$ and $B$ is denoted by  $\ip{A}{B}:={\rm trace}(AB)$.
For two matrices $A,B\in \R^{n\times n}$, $A\geq B$, means $a_{ij}\geq b_{ij}$, for all $i,j$.

We use the notation  $I_n$ (resp., $J_n$) for the identity matrix (resp., the matrix of all ones) of order $n$,
and $e_n$ to denote  the vector of all ones.
If the dimension  of the matrices is clear from the context, we omit the subscript.
We use  $u_i$ to denote the  $i$-th standard basis vector.
Here, we use $[m]$ to denote the set $\{1,\dots,m\}$.

The `${\rm vec}$' operator stacks the columns of a matrix, while the `${\rm diag}$' operator maps an
$n\times n$ matrix to the $n$-vector given by its diagonal. The adjoint operator of `diag' is denoted by `Diag'.
The Kronecker product $A \otimes B$ of matrices $A \in \R^{p \times q}$ and $B\in \R^{r\times s}$ is
defined as the $pr \times qs$ matrix composed of $pq$ blocks of size $r\times s$, with block $ij$ given by
$A_{i,j}B$, $i = 1,\ldots,p$, $j = 1,\ldots,q$.

\section{Problem formulation} \label{sec:form}

Let $G=(V,E)$ be a simple  undirected  graph with the vertex set $V$ and the edge set $E$. Let $|V|=n$, and let $k$ be a given integer such that $1\leq k \leq n-1$.
We say that $G$ is $k$-colorable if one can assign to each vertex in $G$ one of the $k$ colors such that adjacent vertices do not have the same color.
A graph $G'=(V',E')$ is called an induced subgraph of a given graph $G=(V,E)$ if $V'\subseteq V$ and $E'\subseteq E$  is the set of all edges in $E$ connecting the vertices in $V'$.

The maximum $k$-colorable subgraph problem  is to find  an induced $k$-colorable subgraph  with maximum cardinality in the given graph $G$.
 We denote by $\alpha_k(G)$ the number of vertices in the maximum $k$-colorable subgraph of $G$.
 When $k=1$, the  M$k$CS problem corresponds to the stable set problem, i.e., $\alpha_1(G) = \alpha(G)$, where $\alpha(G)$ denotes
  the stability number of $G$ (the maximum cardinality of a  stable set in $G$).

For any $k \ge 1$, the M$k$CS problem can be formulated as an integer  problem.
Let $X\in \{0,1\}^{n\times k}$ be the matrix with one in the entry $(i,r)$ if vertex $i\in [n]$ is colored with color $r \in [k]$  and zero otherwise.
An IP formulation for the   M$k$CS problem is given by \eqref{obj1}--\eqref{prod:con2}:
\begin{subequations}
\begin{align}
\alpha_k(G)= \max_{X\in \{0,1\}^{n\times k}} & \ \sum_{i\in [n],r\in [k]} X_{ir}  \label{obj1} \\
 \st \hspace{0.5cm}& \ X_{ir}X_{jr} = 0,  \ \text{ for all } \lrc{ij}\in E, r\in [k]  \label{prod:con1} \\
 & \ \sum_{r\in [k]} X_{ir}\le 1, \  \text{ for all } i\in [n].  \label{prod:con2}
\end{align}  \label{prodModel}
\end{subequations}
Here, constraints \eqref{prod:con1} ensure that two adjacent vertices are not colored with the same color, and constraints \eqref{prod:con2}
that each vertex is colored with at most one color.

Alternative IP formulation for  the  M$k$CS problem can be obtained by adding a binary slack variable to each inequality constraint in \eqref{prod:con2},
i.e., by  replacing those constraints with  the following ones:
\begin{equation} \label{con:EQ}
 \sum_{r\in [k{+}1]}X_{ir}=1, \quad \text{ for all } i \in [n].
 \end{equation}
It is known that a model with equality constraints may provide stronger relaxations than alternative one, see e.g.,  \cite{Burer09,RendlSotirov}.
In Section \ref{sec:ipv} we use the IP model with equality constraints \eqref{con:EQ} to derive our strongest SDP relaxation for the  M$k$CS problem.

Another IP model for the  M$k$CS problem is obtained by replacing  \eqref{prod:con1} constraints with  the following ones:
\begin{equation} \label{con:Jan}
X_{ir}+X_{jr} \le 1,  \quad  \lrc{ij}\in E, ~r\in [k].
\end{equation}
The IP formulation \eqref{obj1}, \eqref{prod:con2}, \eqref{con:Jan}  is exploited in \cite{JanuschowskiBCP,JanuschowskiOrb}.

\section{The generalized $\vartheta$- and $\vartheta'$-number} \label{sec:theta}

In this section we present the generalized $\vartheta$-number by \citet{Narasimhan1}, which is an upper bound for the  M$k$CS problem.
 This eigenvalue upper bound was reformulated as an SDP relaxation in  \cite{Alizadeh}.
Here, we strengthen the mentioned SDP relaxation by adding non-negativity constraints and introduce the generalized $\vartheta'$-number.

\subsection{Eigenvalue and SDP formulations of the generalized $\vartheta$-number} \label{subsec:thetaEig}

\citet{Narasimhan1} introduce $\vartheta_k(G)$, the generalized $\vartheta$-number, as an upper  bound for $\alpha_k(G)$:
\begin{align}
\alpha_k(G)\le \vartheta_k(G)=\min_{A\in \Sym^n}
\left \{ \sum_{i=1}^k \lambda_i(A): \ A_{ij}=1 \text{ for } \lrc{ij} \notin E \mbox{ or } i=j
\right \}, \label{pr:theta1}
\end{align}
where $\lambda_1(A)\ge \lambda_2(A) \ge \dots \ge \lambda_n(A)$ are eigenvalues of $A$.
The bound $\vartheta_k(G)$ improves the obvious bound $\alpha_k(G)\le k \vartheta(G)$.
The minimum in \eqref{pr:theta1} is attained, see Theorem 12 in \citet{colorThesis}.
To show that $\vartheta_k(G)$ is an upper bound on $\alpha_k(G)$, we follow the reasoning of~\citet{Mohar} who use Fan's theorem.

\begin{theorem}[\citet{Fan1}] \label{thm:fan} Let $A$ be a symmetric matrix with eigenvalues $\lambda_1(A)\ge \lambda_2(A) \ge \dots \ge \lambda_n(A)$. Then
\begin{align}
\sum_{i=1}^k \lambda_i(A) =\max_{X\in\R^{n\times k}} \left \{ \ip{A}{XX\tr } \ \st \ X\tr X=I_k \right \}. \label{pr:1}
\end{align}
\end{theorem}

Let  $X^*\in \{0,1\}^{n\times {k}}$ be an optimal solution to the IP problem~\eqref{prodModel}. That is, for every $r\in [k]$ the $r^{th}$ column of $X^*$ is the incidence vector of the stable set colored in color $r$.
Let $\hat{X}$ be the matrix whose columns are the columns of $X^*$ normalized to one. Thus,  $\hat{X}\tr \hat{X}=I_k$ by construction.
Now, using Theorem~\ref{thm:fan},  for any matrix $A$ feasible for problem~\eqref{pr:theta1} we have
\[
\alpha_k(G)=\ip{A}{\hat{X}\hat{X}\tr}\le \max_{X\in\R^{n\times k}} \left \{ \ip{A}{XX\tr } \ \st \ X\tr X=I_k \right \} =\sum_{i=1}^k \lambda_i(A).
\]
Hence $\alpha_k(G) \le \vartheta_k(G)$.
For $k=1$, $\vartheta_k(G)$ is the eigenvalue formulation of the $\vartheta$-number by \citet{Lovasz}.
In the original paper by \citet{Narasimhan1}, the generalized $\vartheta$-number is introduced for the generalized clique number of $G$, which implies that the generalized $\vartheta$-number there
is defined for the complement of $G$.
The  generalized $\vartheta$-number in \cite{Narasimhan1} is introduced by showing that the sum of the $k$ largest  eigenvalues of the adjacency matrix (with ones along the diagonal) of a graph is at least as large as the size of the largest induced subgraph that can be covered with $k$ cliques.
This result is a generalization of the well known result that the largest eigenvalue of the adjacency matrix (with ones along the diagonal) of a graph is
at least as large as the size of the largest clique in the graph.
Here, we are interested in the maximum number of vertices in a $k$-partite subgraph of $G$  and therefore define $\vartheta_k(G)$ for $G$. \\

It is known that problem~\eqref{pr:1} can be formulated as an SDP relaxation, see \citet{Alizadeh}.
This implies that $\vartheta_k(G)$ can be obtained as the optimal solution of the following SDP relaxation:
\begin{align}
 \vartheta_k(G) = \min_{Y \in \Sym^n,~ \mu, ~ \{x_{ij}\}_{\lrc{ij}\in E}} & \ \ip{I}{Y}+\mu k \label{pr:thetak} \\
 \st \hspace{0.95cm}& \ \sum_{\lrc{ij}\in E}\mathcal{E}^{ij}x_{ij}-J+\mu I + Y \succeq 0, \ Y \succeq 0, \nonumber
\end{align}
where  $\mathcal{E}^{ij}= u_i u_j\tr + u_j u_i\tr$ for $i,j\in [n]$ such that $i\neq j$.
The dual of the above problem is
\begin{align}
 \vartheta_k(G) = &  \max_{Z \in \Sym^n}   \ \ip{J}{Z}\label{pr:theta4}   \\
                  & \st  ~~\ Z_{ij}=0 \text{ for } \lrc{ij}\in E, \nonumber  \\
                  & \hspace{0.95cm} \ \ip{I}{Z}=k, \nonumber \\
                  & \hspace{0.95cm} \ Z \succeq 0, \ I-Z \succeq 0. \nonumber
\end{align}
Note that $Z=\tfrac{k}{n}I$ is a strictly feasible point for the SDP relaxation \eqref{pr:theta4}, in particular strong duality holds for the primal-dual pair \eqref{pr:thetak}--\eqref{pr:theta4}.

\medskip
Alternatively, the SDP relaxation~\eqref{pr:theta4} can be obtained directly from the IP model \eqref{prodModel}.
As before, let  $X\in \{0,1\}^{n\times {k}}$ be an optimal solution to problem~\eqref{prodModel}, and let $\hat{X}$ be the matrix whose columns are the columns of $X$ normalized to one.
Then the matrix $Z:=\hat{X}\hat{X}\tr$ is feasible for the relaxation \eqref{pr:theta4} with the objective value $\alpha_k(G)$.
In particular,  this follows since for any $r= 1,\dots,k$ we have $\big (\hat{X}\hat{X}\tr\big )_{ij}=\tfrac{1}{c_r}$ if vertices $i$ and $j$ are colored with color $r$, where $c_r$ is the total number of vertices colored in color $r$; $\big (\hat{X}\hat{X}\tr\big )_{ij}=0$ otherwise.
The first constraint in \eqref{pr:theta4} is satisfied by construction of $\hat{X}$ and clearly $\hat{X}\hat{X} \succeq 0$.
The second and the last constraints in \eqref{pr:theta4} are satisfied since the columns of $\hat{X}$ are normalized to one,  $\hat{X}\hat{X}\tr$ has $k$ eigenvalues equal to one, and $n-k$ eigenvalues equal to zero.

For $k=1$ constraint $I-Z\succeq 0$ in \eqref{pr:theta4} becomes redundant since $Z$ is positive semidefinite and its eigenvalues sum up to one.
In this case the SDP relaxation \eqref{pr:theta4} reduces to the following formulation of the $\vartheta$-number by \citet{Lovasz}:
\begin{align}
 \vartheta(G)=\max_{Z \in \Sym^n} & \ \ip{J}{Z} \label{pr:thetaLov2} \\
 \st \hspace{0cm}&  \ Z_{ij}=0 \text{ for } \lrc{ij}\in E, \nonumber \\
 & \ \ip{I}{Z}=1, \nonumber \\
& \ Z \succeq 0 . \nonumber
\end{align}

\subsection{Strengthening the generalized $\vartheta$-number} \label{subsec:prime}

Note that all entries of an optimal solution to the integer programming problem \eqref{prodModel} are non-negative.
Therefore, we can add  non-negativity constraints to the matrix variable in \eqref{pr:theta4}  to strengthen the relaxation.
This leads to the following SDP relaxation:
\begin{align}
\vartheta'_k(G) = \max_{Z\in \Sym^{n}} & \ \ip{J}{Z} \label{pr:theta'} \\
 \st \hspace{0.2cm}  & \ Z_{ij} = 0 \text{ for } \lrc{ij}\in E \nonumber \\
  & \ \ip{I}{Z}=k    \nonumber \\
 & \ Z\succeq 0, \ I-Z\succeq 0, \nonumber\\
  & \ Z\ge 0.  \nonumber
\end{align}
We  refer to the solution of the above SDP relaxation as the generalized $\vartheta'$-number.
Note that for $k=1$, $\vartheta'_k(G)$ equals the $\vartheta'(G)$ upper bound on $\alpha(G)$ by~\citet{SchrijverTheta}.

\section{Vector lifting SDP relaxations} \label{sec:ipv}

In this section we derive two new SDP relaxations for the M$k$CS problem. To derive relaxations we use the vector lifting approach, see e.g., \cite{DingGeWolk,WolkZhao,kCut2}.
In Section \ref{sec:SymColorV}, we reduce the SDP models by exploiting the fact that the M$k$CS problem is invariant under any permutation of the colors.
To strengthened the relaxations, we add inequalities from the boolean quadric polytope in Section \ref{subsec:bqp}.
 In Section \ref{sec:alpha} we relate the stable set   problem on the Cartesian product of the complete graph on $k$ vertices and $G$, with the M$k$CS problem.
  In particular, we show that our weaker vector lifting relaxation is equivalent to  the SDP model for the Schrijver $\vartheta'$-number on the Cartesian product of two mentioned graphs.

To derive our first SDP relaxation in this section, we consider the IP model for the M$k$CS problem with all equality constraints.
Thus, assume that $X\in \{0,1\}^{n\times (k{+}1)}$ is feasible for the IP model  \eqref{obj1}, \eqref{prod:con1} and \eqref{con:EQ}, i.e.,
 $X$ is the matrix with one in the entry $(i,r)$ if  vertex $i$ is colored with color $r$ and zero otherwise, where `color' $k{+}1$ represents the uncolored vertices. Let $x=\vecc (X)$ and $Y = xx\tr$. As $x \in \{0,1\}^{n(k+1)}$, it follows that  $Y={\rm diag}(Y){\rm diag}(Y)\tr$, which  can be relaxed to  $Y-{\rm diag}(Y){\rm diag}(Y) \tr \succeq 0$ which is equivalent to the following convex constraint:
\begin{align}
 \begin{bmatrix}
1& {\rm diag}(Y)\tr \\
{\rm diag}(Y) & Y
\end{bmatrix} \succeq 0.
\end{align}
Matrix $Y$   consists of $(k+1)^2$ blocks of the size $n{\times}n$. We denote by $Y^{rl}$ the $n{\times}n$ block of $Y$ located in position $(r,l) \in [k+1]{\times}[k+1]$.
From here on we  use the subscripts $i,j$ to indicate vertices and use superscripts $r,l$ to indicate colors in matrix variables.
From \eqref{prod:con1},  there follows
\[
 Y^{rr}_{ij}=0,  \quad \forall  \{ij\}\in E, \ r \in [k],
 \]
from  \eqref{con:EQ},
\[
\sum_{r\in [k{+}1]} Y_{ii}^{rr} = 1, \quad \text{ for all } i\in [n].
\]
Also, from  \eqref{con:EQ} and the fact that $X$ is binary, we have
\[
\ Y^{rl}_{ii} = 0, \quad \text{ for all } i\in [n], \  r,l\in [k{+}1], \ r\neq l.
\]
We collect  all above listed constraints, add non-negativity constraints, and  arrive to the following SDP relaxation for the M$k$CS  problem:
\begin{align}
\max_{Y\in \Sym^{n(k{+}1)}}    &  \ \ \sum_{r\in [k]} \sum_{i\in [n]}Y^{rr}_{ii} \label{pr:sdp1}\\
\st \hspace{0.4cm} & \  Y^{rr}_{ij}=0, \ \text{ for all } \{ij\}\in E, \ r \in [k] \nonumber \\
&\sum_{r\in [k{+}1]} Y_{ii}^{rr} = 1, \quad \text{ for all } i\in [n] \nonumber \\
  & \ Y^{rl}_{ii} = 0, \ \text{ for all } i\in [n], \  r,l\in [k{+}1], \ r\neq l \nonumber \\
& \ Y\ge 0 , \ \begin{bmatrix}
1& \text{diag}(Y)\tr  \\
\text{diag}(Y) &Y
\end{bmatrix} \succeq 0. \nonumber
\end{align}
The  SDP relaxation \eqref{pr:sdp1} is not strictly feasible.
Namely, one can verify that the columns of $[-e_n, e_{k+1}\tr \otimes I_n ]\tr$ are contained in the nullspace
of the barycenter point, that is a point in the relative interior of the minimal face containing the feasible set of \eqref{pr:sdp1}, see e.g.,~\cite{WolkZhao}.
Similarly, we derive  a vector lifting SDP relaxation from the IP problem~\eqref{prodModel}.
Namely, by exploiting  $X\in \{0,1\}^{n\times k}$ that is feasible for \eqref{prodModel},  we obtain  the following SDP relaxation:
\begin{subequations}
\begin{align}
 \max_{Y\in \Sym^{nk}}    &  \ \ \sum_{r\in [k]} \sum_{i\in [n]}Y^{rr}_{ii}   \label{pr:sdp2obj} \\ 
\st \hspace{0.1cm} & \  Y^{rr}_{ij}=0, \ \text{ for all } \{ij\}\in E, \ r \in [k]  \label{eq:edge} \\
  & \ Y^{rl}_{ii} = 0, \ \text{ for all } i\in [n], \  r,l\in [k], \ r\neq l  \label{eq:orthog} \\
& \ Y\ge 0, \ \begin{bmatrix}
 1 & \text{diag}(Y)\tr \\
\text{diag}(Y)  & Y
\end{bmatrix} \succeq 0. \label{ineq:SDP1}
\end{align}  \label{pr:sdp2}
\end{subequations}
Note that it is not necessary to include in \eqref{pr:sdp2} the   constraints  $\sum_{r\in [k]} Y_{ii}^{rr} \le  1$ ($i\in [n]$) that arise naturally  from \eqref{prod:con2},
as they are redundant.
Namely, we have the following result.
\begin{lemma} \label{diagredundant}
For each $i\in [n]$, the constraint $\sum_{r\in [k]} Y_{ii}^{rr} \le  1$  is redundant for the SDP relaxation \eqref{pr:sdp2}.
\end{lemma}

\begin{proof}
Let $Y$ be feasible for \eqref{pr:sdp2} and  $i \in [n]$.
Consider the principal submatrix of $Y$, $Y(\alpha)$,   of the size $(k+1)\times (k+1)$ whose rows and columns are indexed by
$\alpha=\{1, i+1, n+i+1, \ldots, (k-1)n+i+1 \}$.
From $Y^{rl}_{ii}=0$ for $r,l\in [k]$ such that $r\neq l$, it follows that
\[
Y(\alpha) = \begin{bmatrix}
1 & {y}\tr \\
{y} & {\rm Diag} ({y})
\end{bmatrix}, \mbox{ where } {y} = [Y^{11}_{ii},Y^{22}_{ii}, \ldots, Y^{kk}_{ii}].
\]
Since  $Y(\alpha) \succeq 0$, we have
\[
e\tr {\rm Diag} ({y}) e - (e\tr y)(y\tr e)= \sum_{r=1}^k y_r - \left(\sum_{r=1}^k y_r \right)^2 \geq 0,
\]
which implies $\sum_{r\in [k]} Y^{rr}_{ii}\le 1$.
\end{proof}
The SDP relaxation \eqref{pr:sdp2} is strictly feasible. To show this, one can use an argument that is  similar
to the one that shows that the symmetry reduced SDP relaxation
\eqref{pr:sdpSymColor2}  is strictly feasible, see Lemma \ref{lem:SymColor} and Theorem \ref{thm:SymColor2}.
It is not difficult to verify that the  SDP relaxation \eqref{pr:sdp2}  is dominated by the SDP relaxation \eqref{pr:sdp1}.
We show below that  those two relaxations are equivalent after adding the following inequality constraints to the relaxation \eqref{pr:sdp2}:
\begin{align}
& \ 1{-}\sum_{r\in [k]}\hspace{-0.1cm}  Y_{ii}^{rr}{-}\sum_{r\in [k]}\hspace{-0.1cm} Y_{jj}^{rr}{+}\sum_{r\in [k]}\sum_{l\in [k]} \hspace{-0.1cm} Y_{ij}^{rl} \ge 0, \text{ for all } i> j \label{ineq:11} \\
&\ Y_{ii}^{ll}-\sum_{r\in [k]} Y_{ij}^{rl} \ge 0, \ \text{ for all } i\neq j, \ l, \label{ineq:22}
\end{align}
where $i,j\in [n]$ and $l,r\in [k]$.
These inequalities  are based on the reformulation-linearization technique by \citet{SherAdams}.
In particular, inequalities~\eqref{ineq:11} are linearizations of the products of pairs of constraints \eqref{prod:con2}.
Inequalities~\eqref{ineq:22} represent multiplication of elementwise non-negativity constraint on $X$ with each individual constraint  in~\eqref{prod:con2}.
Similar inequalities are  used in \cite{RendlSotirov} and \cite{RenSoTur} to improve  SDP relaxations for the min-cut problem and the bandwidth problem, respectively.

\begin{theorem}\label{thm:equiv0}
The SDP relaxation \eqref{pr:sdp2} with additional constraints~\eqref{ineq:11},~\eqref{ineq:22} is equivalent to the SDP relaxation \eqref{pr:sdp1}.
\end{theorem}
\begin{proof} (See also \cite{RendlSotirov}, Section 5.2.)
Let $Z$ be feasible for problem~\eqref{pr:sdp2} and assume it also satisfies  \eqref{ineq:11},~\eqref{ineq:22}.
We define $z:=\text{diag}(Z)$. For ease of presentation we write constraints~\eqref{ineq:11} and~\eqref{ineq:22} as matrix inequalities using the following transformation matrix:
\begin{align*}
M=\matr{1& -(m^1)\tr \\ \dots&\dots \\1& -(m^n)\tr}, \text{ where  }
 m^i_j=\begin{cases} 1,& j\in \{i,n+i,\dots,n(k{-}1)+i\}\\ 0,& \text{ otherwise}.  \end{cases} 
\end{align*}
Then inequalities~\eqref{ineq:11} can be written as follows:
\begin{align*}
&M\matr{1&z \tr \\
 z& Z}M\tr \ge 0,
\end{align*}
and inequalities~\eqref{ineq:22}    as follows:
 \begin{align*}
 &  M\matr{z \tr \\
 Z}\ge 0.
\end{align*}
Now, it is not difficult to verify that
\[
Y=\matr{I\\M}\matr{1&z \tr \\
 z& Z}\matr{I&M \tr}\succeq 0
 \]
is feasible for the SDP relaxation \eqref{pr:sdp1}.

Conversely, let $Y$ be feasible for the SDP relaxation \eqref{pr:sdp1}. Then $Z:=Y(1{:}kn,1{:}kn)$ is feasible for the SDP  relaxation~\eqref{pr:sdp2}
by construction of both relaxations. To show that $Z$ also satisfies  \eqref{ineq:11} and ~\eqref{ineq:22}
note that that those inequalities correspond to the missing blocks  $Y^{(k+1)r}$ for $r\in [k+1]$.
\end{proof}

To strengthen relaxations \eqref{pr:sdp1} and \eqref{pr:sdp2}, one may tend to add the clique constraints i.e.,
$\sum_{i \in C} {\rm diag}(Y^{rr}_{ii}) \leq 1$  where $C\subseteq [n]$ denotes a set of indices corresponding to vertices in a clique, and $r\in [k]$.
However, those constraints are redundant, see Remark \ref{cycleRedundant}, page \pageref{cycleRedundant} for an explanation.

\subsection{Symmetry reduction on colors} \label{sec:SymColorV}

In this section we exploit the fact that the M$k$CS problem is invariant under  permutations of the colors,
in order to reduce the sizes of  the vector  lifting SDP relaxations from the previous section.
We also show that the symmetry-reduced SDP relaxations  are  strictly feasible.

We begin with a lemma related to  strict feasibility.
\begin{lemma}\label{lem:SymColor}
Let $n\ge 1, k\ge 1$, and let $e\in \R^n$ be the vector of all ones. Then
\begin{align*}
M:=\matr{k& \tfrac{1}{(n+1)}e\tr \\
\tfrac{1}{(n+1)}e& \tfrac{1}{(n+1)}I }\succ 0. 
\end{align*}
\begin{proof}
Let ${\tiny \matr{x_0\\ x}} \in \R^{n+1}\setminus \{0\}$ we have
  \begin{align*}
{\tiny \matr{x_0\\ x}}\tr M {\tiny \matr{x_0\\ x}}&=\frac 1{n+1}\left( ((n+1)k-n)x_0^2 + \|x_0e + x\|^2\right) > 0
   \end{align*}
\end{proof}

\end{lemma}

\noindent
To reduce the size of the SDP relaxation \eqref{pr:sdp2} with additional constraints \eqref{ineq:11} and \eqref{ineq:22}, we need  the following result.
\begin{lemma}[Lemma 2.8 in \cite{GvozdenLaurent}]\label{lem:SymColorTrans}
Let $Y\in  {\R^{kn \times kn}}$ be a  block matrix that consists of $k^2$ blocks of the size $n{\times}n$. Let $Y$ have a matrix $A\in \Sym^{n}$
as its diagonal blocks, and a matrix $B\in \Sym^{n}$ as its non-diagonal blocks, i.e.,
\[
Y=\underbrace{\matr{A&B&\dots&B\\
B&A&\dots&B\\
\vdots&\vdots&\ddots&\vdots\\
B&B&\dots&A
}}_{k \text{ blocks}}=I\otimes  A + (J-I)\otimes B.
\]
Then $Y\succeq 0$ if and only if $A-B\succeq 0$ and $A+(k-1)B\succeq 0$.
\end{lemma}

Now, we are ready to prove our main result in this section.

\begin{theorem}\label{thm:SymColor2}
The SDP relaxation \eqref{pr:sdp1} and the SDP relaxation \eqref{pr:sdp2} with additional constraints \eqref{ineq:11},~\eqref{ineq:22} are equivalent to the following SDP relaxation:
\begin{subequations}
\begin{align}
\theta_k^1(G) = \max_{Z,X\in \Sym^{n}} & \ \ip{I}{Z} \label{pr:sdpSymColor2obj} \\
 \st \hspace{0.3cm}   & \ Z_{ij} = 0, \ \text{ for } \lrc{ij}\in E \nonumber \\
   & \ X_{ii} = 0, \ \text{ for } i\in [n]  \label{eq:orthogSymColor} \\
  & \ Z\ge 0, \ X\ge 0  \label{nonegatXZ} \\
  & \ Z-X\succeq 0 \label{ineq:symColor1} \\
 & \ \matr{1& {\rm diag}(Z)\tr  \\
{\rm diag}(Z) & Z+(k-1)X }\succeq 0 \label{ineq:symColor2}\\
& \ 1{-}Z_{ii}{-}Z_{jj}{+}Z_{ij}{+}(k-1)X_{ij} \ge 0,\  \text{ for } i,j\in [n],i>j  \label{ineq:symColor3}\\
& \ Z_{ii}{-}Z_{ij}{-}(k-1)X_{ij} \ge 0,\  \text{ for } i,j\in [n],i\neq j,  \label{ineq:symColor4}
\end{align} \label{pr:sdpSymColor2}
\end{subequations}
and the above SDP relaxation is strictly feasible.
\end{theorem}
\begin{proof}  First, the relaxations \eqref{pr:sdp1} and~\eqref{pr:sdp2}  with additional constraints~\eqref{ineq:11},~\eqref{ineq:22} are equivalent by Theorem~\ref{thm:equiv0}.
We show here that the SDP relaxation \eqref{pr:sdp2}  with~\eqref{ineq:11},~\eqref{ineq:22} is equivalent to~\eqref{pr:sdpSymColor2}.

Let $Y$ be a feasible solution to the SDP relaxation \eqref{pr:sdp2}  with additional constraints~\eqref{ineq:11},~\eqref{ineq:22}.
If we permute the color labels, we permute the ``columns" and ``rows" of blocks in $Y$. For instance, permuting color $r$ and color $l$ results in permuting blocks $Y^{pl}$ and $Y^{pr}$, for all $p\in [k]$,
and then permuting blocks $Y^{lp}$ and $Y^{rp}$, for all $p\in [k]$.  In particular $Y^{rr}$ and $Y^{ll}$ are permuted.

Let $\bar Y$ be the average over all $k!$ permutations of the color labels.
By construction problem~\eqref{pr:sdp2}  with constraints~\eqref{ineq:11},~\eqref{ineq:22} is convex and invariant under  color permutations.
Therefore,  $\bar Y$ is feasible for \eqref{pr:sdp2}  and satisfies the additional constraints~\eqref{ineq:11},~\eqref{ineq:22}.
Notice that $\bar Y$ has the form
\begin{equation}\label{def:Yinv}
\bar{Y}=\frac{1}{k}\underbrace{\matr{Z& X &\dots& X\\
X & Z  &\dots& X\\
\vdots&\vdots&\ddots&\vdots\\
X  & X &\dots& Z
}}_{k \text{ blocks}}=\tfrac{1}{k}I_k\otimes   Z+\tfrac{1}{k}(J_k{-}I_k)\otimes X,
\end{equation}
where $Z = \sum_{r=1}^k Y^{rr}$ and $X = \frac 1{k-1} \sum\limits_{r,l=1, r\neq l}^k  Y^{rl}$.

The SDP relaxation \eqref{pr:sdp2} with constraints~\eqref{ineq:11},~\eqref{ineq:22} can be restricted without loss of generality to matrices  of the form \eqref{def:Yinv}, which results in \eqref{pr:sdpSymColor2}.
Indeed, the objective and all linear constraints of the SDP relaxation \eqref{pr:sdpSymColor2} are obtained by rewriting the objective and the corresponding linear constraints of \eqref{pr:sdp2}, \eqref{ineq:11} and \eqref{ineq:22} for  matrices  of the form \eqref{def:Yinv}.
Now, consider the SDP constraint
\[
\begin{bmatrix}
1& \text{diag}(Y)\tr\\
\text{diag }Y &Y
\end{bmatrix} \succeq 0,
\]
where $Y$ is of the form \eqref{def:Yinv}
which, by the Schur complement, is equivalent to $Y-\text{diag}(Y)\text{diag}(Y)\tr \succeq 0$.
We have
\[
Y-\text{diag}(Y) \text{diag}(Y)\tr  =\tfrac{1}{k} I\otimes  ( Z-\tfrac{1}{k} \text{diag}(Z)\text{diag}(Z)\tr)
  + \tfrac{1}{k}(J-I)\otimes (X-\tfrac{1}{k}{\text{diag}(Z)\text{diag}(Z)\tr}).
\]
Hence by Lemma~\ref{lem:SymColorTrans}, the positive semidefinite  constraint holds if and only if
\[
Z-X\succeq 0 \text{ and } \ Z+(k-1)X-\text{diag}(Z)\text{diag}(Z)\tr \succeq 0.
\]

Now we show strict feasibility of~\eqref{pr:sdpSymColor2}.
Let $A_{\bar{G}}$ be the adjacency matrix of the complement of $G$ and  $M \succ 0$ be given as in Lemma~\ref{lem:SymColor}.
Since also $\tfrac{1}{k(n+1)}I \succ 0$, there exists $0<\eps<\tfrac{1}{kn(n+1)}$ such that
\begin{align*}
\hspace{-1cm}\tfrac{1}{k}M+\eps \matr{0&0\tr \\
0&A_{\bar{G}}+(k-1)(J-I)} \succ 0 \ \text{ and } \ \tfrac{1}{k(n+1)}I+\eps A_{\bar{G}}-\eps(J-I) \succ 0.
\end{align*}
Define $\bar{Z}:=\tfrac{1}{k(n+1)}I+\eps A_{\bar{G}}, \ \bar{X}:=\eps(J-I)$.
By the choice of $\eps$   one can verify
that $(\bar{Z},\bar{X})$ strictly satisfies  constraints  \eqref{nonegatXZ}, \eqref{ineq:symColor3} and \eqref{ineq:symColor4}.
For example, we verify below  that constraints~\eqref{ineq:symColor3} are strictly satisfied for $(\bar{Z},\bar{X})$:
\begin{align*}
1{-}\bar{Z}_{ii}{-}\bar{Z}_{jj}{+}\bar{Z}_{ij}{+}(k-1)\bar{X}_{ij}\ge 1-\tfrac{2}{k(n+1)}=\tfrac{kn+k-2}{k(n+1)}\ge \tfrac{1}{k(n+1)},
\end{align*}
for  all $i,j\in [n],i>j$.
 \end{proof}

\begin{remark} \label{rem:theta2}
We denote by $\theta_k^2 (G)$ the optimal value of the SDP relaxation obtained from
\eqref{pr:sdpSymColor2} where constraints \eqref{ineq:symColor3} and \eqref{ineq:symColor4} removed.
Thus, $\theta_k^2 (G)$ equals    the optimal value of the SDP relaxation \eqref{pr:sdp2}.
\end{remark}
For the sake of completeness, we also present the symmetry-reduced version of the relaxation \eqref{pr:sdp1}.
\begin{corollary}\label{cor:SymColor3}
The SDP relaxation \eqref{pr:sdp1}
is equivalent to the following relaxation:
\begin{align}
\max_{Z,X,V,W\in \Sym^{n }} & \ \ip{I}{Z} \nonumber \\ 
 \st \hspace{0.1cm}   & \ Z_{ij} = 0 \text{ for } \lrc{ij}\in E \nonumber \\
    & \ X_{ii} = 0, \ \text{ for } i\in [n]  \nonumber \\
        & \ W_{ii} = 0, \ \text{ for } i\in [n]  \nonumber \\
  & \ Z_{ii} + \frac{1}{k}V_{ii} = 1 \text{ for } i\in [n]  \nonumber \\
  & \ Z\ge 0, \ X\ge 0, \ V \ge0, \ W\ge 0  \nonumber \\
  & \ Z-X\succeq 0 \nonumber \\
 &  \matr{k& \sqrt{k} ~{\rm diag}(Z)\tr & {\rm diag}(V)\tr\\
\sqrt{k}{\ \rm diag}( Z) & Z+(k-1)X & W \\
{\rm diag}(V) & W & V }\succeq 0. \nonumber
\end{align}
\end{corollary}
\begin{proof} First, relaxations \eqref{pr:sdp1} and  \eqref{pr:sdpSymColor2} are equivalent by Theorem~\ref{thm:equiv0} and Theorem~\ref{thm:SymColor2}.
Now, consider a feasible solution $Y$ to the SDP relaxation \eqref{pr:sdp1}.
The problem is invariant under permutations of $k$ colors, which correspond to the first $k^2$ blocks of $Y$, i.e., $Y^{ij}$, ($i,j=1,\ldots,k$).
By invariance of the problem, it is enough to consider the solutions of the form
\[
Y=\frac{1}{k}\matr{ I\otimes   Z + (J-I)\otimes X & e  \otimes W \\
 (e \otimes  W)\tr & V
}.
\]
Further we use the same approach as in the proof of Theorem~\ref{thm:SymColor2}, so the details are omitted.
\end{proof}
Note that the SDP relaxation from Corollary \ref{cor:SymColor3}  has a matrix variable of order $3n$, while the SDP relaxation from Theorem \ref{thm:SymColor2} has
a matrix variable of order $2n$. However, they are equivalent.

\subsection{Boolean quadric polytope inequalities} \label{subsec:bqp}

To further strengthen our strongest SDP relaxation,
one can add inequalities from the boolean quadric polytope (BQP), see e.g., \citet{Padberg1989}.
Namely, let $X$ be a feasible solution to the binary problem~\eqref{prodModel} and consider $Y=\vecc(X) \vecc (X) \tr.$
Then for all $i,j,p\in [nk]$ the following BQP inequalities are valid for $Y$:
\begin{align}
&0\le Y_{i,j} \le Y_{i,i} \label{ineq:411}\\
&Y_{i,i}+ Y_{j,j}\le 1+Y_{i,j} \label{ineq:422}\\
& Y_{i,p}+  Y_{j,p}\le Y_{p,p} +Y_{i,j}  \label{ineq:311}\\
& Y_{i,i}+  Y_{j,j} + Y_{p,p} \le Y_{i,j}+Y_{i,p}+  Y_{j,p}+1. \label{ineq:322}
\end{align}
Therefore, one can add those constraints to the SDP relaxation \eqref{pr:sdp2} with additional constraints~\eqref{ineq:11}, \eqref{ineq:22} in order to further strengthen it.
In the view of the symmetry reduction on colors, we may consider only those feasible solutions $Y$ to the relaxation  \eqref{pr:sdp2} that can be written as in~\eqref{def:Yinv}.
Then, \eqref{ineq:411}--\eqref{ineq:322} reduce to
\begin{align}
&0\le Z_{i,j} \le Z_{i,i}, \quad 0\le X_{i,j} \le Z_{i,i}, \label{ineq:41}\\[7pt]
&Z_{i,i}+ Z_{j,j}\le k+Z_{i,j}, \quad Z_{i,i}+  Z_{j,j}\le k+X_{i,j}, \label{ineq:42} \\[7pt]
& X_{i,p}+  X_{j,p}\le Z_{p,p} +X_{i,j}, \quad   Z_{i,p}+  Z_{j,p}\le Z_{p,p} +Z_{i,j}, \quad   X_{i,p}+  X_{j,p}\le Z_{p,p} +Z_{i,j}, \label{ineq:31}\\
& X_{i,p}+  Z_{j,p}\le Z_{p,p} +X_{i,j}, \quad  Z_{i,p}+  X_{j,p}\le Z_{p,p} +X_{i,j},    \nonumber\\[7pt]
& Z_{i,i}+  Z_{j,j} + Z_{p,p} \le X_{i,j}+X_{i,p}+  X_{j,p}+k, \quad Z_{i,i}+  Z_{j,j} + Z_{p,p} \le Z_{i,j}+Z_{i,p}+  Z_{j,p}+k,\label{ineq:32}  \\
& Z_{i,i}+  Z_{j,j} + Z_{p,p} \le Z_{i,j}+X_{i,p}+  X_{j,p}+k, \quad  Z_{i,i}+  Z_{j,j} + Z_{p,p} \le X_{i,j}+X_{i,p}+  Z_{j,p}+k,\nonumber \\
& Z_{i,i}+  Z_{j,j} + Z_{p,p} \le X_{i,j}+Z_{i,p}+  X_{j,p}+k, &\nonumber
\end{align}
where $i,j,p\in [n], i\neq j\neq p$.
Inequalities~\eqref{ineq:41} correspond to~\eqref{ineq:411}, \eqref{ineq:42} correspond to~\eqref{ineq:422}, etc.

Some of the BQP constraints~\eqref{ineq:41}--\eqref{ineq:32} are redundant for the relaxation \eqref{pr:sdpSymColor2}.
In particular,~\eqref{ineq:41} follows from~\eqref{ineq:symColor4}. Also, constraints \eqref{ineq:42} are redundant since the $2$-clique constraints are redundant for the SDP relaxations~\eqref{pr:sdp1} and~\eqref{pr:sdp2}.
Therefore, in numerical experiments we use inequalities of the type~\eqref{ineq:31} and~\eqref{ineq:32}.
Notice that the first inequality in each of the two sets~\eqref{ineq:31} and~\eqref{ineq:32} is only valid for $k\ge 3$.

One could also consider the triangle inequalities for $Y$ in \eqref{pr:sdp2}, i.e., $Y_{ij}+Y_{jp}-Y_{ip}\le 1$ for any $i,j,p\in [n]$.
However, those inequalities follow from \eqref{ineq:311} and the non-negativity of $Y$.

\subsection{The M$k$CS as the maximum stable set problem} \label{sec:alpha}

The M$k$CS problem on a graph $G$  can be  considered as a stable set problem on the Cartesian product of the complete graph on $k$ vertices and  $G$.
This result was proven by \citet{colorThesis}.  For the sake of completeness, we present a short proof.
We also prove that the Schrijver's $\vartheta'$-number on the Cartesian product of mentioned two graphs equals the optimal value of the vector lifting relaxation~\eqref{pr:sdp2}.\\

We denote by $K_k=(V_k,E_k)$, where $V_k=[k]$, the complete graph on $k$ vertices.
The Cartesian product  $K_k \Box G$ of graphs $K_k$  and $G=(V,E)$ is a graph with the vertex set $V_k \times V$ and the edge set
$E_{\Box}$ where two vertices $(u,i)$ and $(v,j)$ are adjacent if  $u=v$ and $(i,j)\in E$  or $i=j$ and  $(u,v)\in E_k$.
The following result shows that the  M$k$CS problem on $G$ corresponds to the stable set problem on $K_k \Box G$.

\begin{theorem} \cite{colorThesis}  \label{thm:alpha} Let $G=(V,E)$, and let $K_k$ be the complete graph on $k$ vertices. Then $\alpha_k(G)=\alpha(K_k \Box G)$.
\end{theorem}
\begin{proof}
First, if $S_1,\dots,S_k$ are disjoint stable sets in $G$, then $\{1 \}\times S_1,\dots, \{k \} \times S_k$ is a stable set in $K_k \Box G$. On the other hand, let $S$ be a stable set in $K_k \Box G$ of the largest cardinality. Then $S$ can be partitioned into $S_1,\dots,S_k$ such that $S_1= \{ 1\}\times \hat{S}_1,\dots, S_k= \{ k \}\times \hat{S}_k$, $\hat{S}_1\subseteq V,\dots,\hat{S}_k\subseteq V$. Moreover, $\hat{S}_1,\dots,\hat{S}_k$ are disjoint since $u \in \hat S_l \cap \hat S_p$ for some $l,p\in [k]$ with $l\neq p$ implies
that there is an edge between $(l,u)$ and $(p,u)$ that is also in the stable set $S$. Hence $\hat{S}_1,\dots,\hat{S}_k$ are disjoint stable sets in $G$.
\end{proof}
The Schrijver's $\vartheta'$-number on $K_k \Box G$  is as follows
\begin{align} \label{pr:thetaProd}
\vartheta'(K_k \Box G) = \max_{Y\in \Sym^{nk}}    &  \ \ \ip{J}{Y} \\
\st \hspace{0.1cm} & \  Y^{rr}_{ij}=0, \ \text{ for all } \{ij\}\in E, \ r \in [k] \nonumber \\
  & \ Y^{rl}_{ii} = 0, \ \text{ for all } i\in [n], \  r,l\in [k], \ r\neq l \nonumber \\
 & \ \ip{I}{Y}=1,  \nonumber\\
& \ Y\succeq 0 \nonumber\\
& \ Y\ge 0, \nonumber
\end{align}
where $Y$ is of the size $nk\times nk$.
The above SDP relaxation follows directly from  \eqref{pr:theta'} and the definition of $K_k \Box G$.
Next, we relate \eqref{pr:thetaProd}  and \eqref{pr:sdp2}.
\begin{lemma}
The Schrijver's $\vartheta'$-number on $K_k \Box G$  equals the optimal value of the vector lifting relaxation~\eqref{pr:sdp2},
i.e.,
\[
\vartheta'(K_k \Box G) = \theta_k^2(G) \ge \alpha_k(G).
\]
\end{lemma}

\begin{proof}
To see this, first, notice that any feasible  solution to the vector lifting relaxation~\eqref{pr:sdp2}
provides a feasible solution to problem~\eqref{pr:thetaProd}.
In particular let  $Y$ be feasible for~\eqref{pr:sdp2}, then it readily follows that $\hat{Y}=\tfrac{1}{\ip{I}{Y}} Y$ is feasible for~\eqref{pr:thetaProd}, see also \cite{GruberRendl}.
Moreover, the SDP constraint in the vector lifting relaxation \eqref{pr:sdp2} implies
\begin{align*}
\ip{J}{\hat{Y}}=\tfrac{1}{\ip{I}{Y}}\ip{J}{Y} \ge\tfrac{1}{\ip{I}{Y}} \ip{J}{\text{diag}(Y)\text{diag}(Y) \tr}= \ip{I}{Y}.
\end{align*}
For the opposite direction, we assume that $\hat Y$ is feasible for~\eqref{pr:thetaProd}.
Now, we use  Theorem 1  from \citet{GalliLetch}, to define $Y\in \Sym^{kn}$ as follows:
\[
Y_{ij}:=\hat{Y}_{ij}\tfrac{\sum_{m=1}^{kn}\hat Y_{im}\sum_{m=1}^{kn}\hat Y_{jm}}{\ip{J}{\hat{Y}} \hat Y_{ii}\hat Y_{jj}} \text{ for } i\neq j, \text{ and }
 Y_{ii} := \tfrac{\big(\sum_{j=1}^{kn}\hat Y_{ij}\big)^2}{\ip{J}{\hat{Y}} \hat{Y}_{ii}}, \text{ for } i\in [kn].
\]
Then, $Y$ is  feasible for~\eqref{pr:sdp2} and satisfies $\ip{I}{Y}\ge \ip{J}{\hat{Y}}$.
Note that the authors of \cite{GalliLetch} and \cite{GruberRendl} consider the SDP relaxations of the Lov\'asz $\vartheta$-number.
\end{proof}
Alternatively,   one can obtain the above result by considering the equivalent SDP relaxation of the Lov\'asz $\vartheta$-number from \cite{GrpLovasSchrij}.

\begin{remark}  \label{cycleRedundant}
It is well known that the clique constraints are redundant for the Lov\'asz $\vartheta$-number,  e.g., Chapter 9 of \cite{GrpLovasSchrij}.
Therefore, the clique constraints are redundant for the SDP relaxation \eqref{pr:thetaProd},  and consequently also for \eqref{pr:sdp2} and \eqref{pr:sdp1}.
\end{remark}

\section{Matrix lifting SDP relaxations} \label{sec:ipm}

In this section  we derive a matrix lifting SDP relaxation for  the M$k$CS problem.
Relaxations obtained by the matrix lifting approach are known to have less variables and constraints than the corresponding relaxations obtained by the vector lifting approach.
However,  relaxations obtained by those two approaches may  be equal, see \citet{DingGeWolk}.
We show here that our matrix lifting SDP relaxation is dominated by the vector lifting relaxations from the previous section.
However, numerical results  show   that the here derived relaxation is  preferable among other bounding approaches for large graphs.
Namely,  the matrix lifting relaxation provides often the same bound as the strongest  vector lifting  bound
  while requiring  significantly less computational effort, see Section \ref{sec:num} for more details.
We also apply symmetry reduction on colors to further reduce the here introduced relaxation.\\

Let $X\in \{0,1\}^{n\times k}$ be a solution to the IP problem \eqref{prodModel} and consider
\[
Y=\matr{I_k\\X}\matr{I_k&X\tr }=\matr{I_k&X\tr \\ X& XX\tr}.
\]
Linearizing the block $XX\tr$, we obtain the following matrix lifting SDP relaxation for \eqref{prodModel}:
\begin{align}
\max_{Z\in \Sym^{n}, ~X\in \R^{n\times k}} & \ \ip{I}{Z} \label{pr:sdp3} \\
 \st \hspace{0.8cm}   & \ Z_{ij} = 0 \text{ for } \lrc{ij}\in E  \nonumber \\ 
  & \ Z_{ii} \le 1 \text{ for } i\in [n] \nonumber   \\  
  & \ Z_{ii}=\sum_{r\in [k]}X_{ir} \text{ for } i\in [n] \nonumber  \\ 
  & \ Z\ge 0, \ X\ge 0 \nonumber   \\   
 & \ \matr{I_k& X\tr  \\
 X&Z}\succeq 0.  \nonumber
\end{align}  
Here, the  positive semidefinite   constraint is imposed on a matrix variable of the size $(k+n) \times (k+n).$
The zero pattern and constraints on the diagonal follow directly from the construction.
The above relaxation has no constraints that ensure that a vertex can be colored by only one color, while vector lifting relaxations have such constraints.

Notice that the constraint ${\rm diag}(Y) \le e$ is redundant when $k=1$, and that  the resulting SDP relaxation corresponds to  one of the SDP relaxations
for the Schrijver  number, see  \cite{GrpLovasSchrij}.

\subsection{Symmetry reduction on colors} \label{sec:SymColorM}

In this section we exploit the invariance of the M$k$CS  problem w.r.t.~color permutations to reduce the number of variables in the SDP relaxation \eqref{pr:sdp3}.

Let $(Z,X)$ be a feasible solution to \eqref{pr:sdp3},
and $\bar{X}$ be the average over all column permutations of $X$.
Since the SDP relaxation \eqref{pr:sdp3} is convex and invariant under  permutations of the colors,  $(Z,\bar{X})$,  is feasible for \eqref{pr:sdp3}.
By construction, all columns of $\bar{X}$ are equal to each other.
Therefore it is sufficient to consider  solutions $(Z,X)$ for  \eqref{pr:sdp3}, where the columns of $X$ are equal to each other.
If we denote a column of $X$ by $x$, then the constraint $Z_{ii}=\sum_{r\in [k]}X_{ir} \text{ for } i\in [n]$ reduces to
 \begin{equation} \label{const:xk}
 {\rm diag}(Z)=kx,
 \end{equation}
 and the SDP constraint  to
 \[
 \matr{I_k& e x^{\tr} \\
x e^{\tr} & Z}\succeq 0,
\]
where $e\in \R^k$.
Now, we use the Schur complement and \eqref{const:xk} to rewrite the SDP constraint:
  \begin{align*}
 Z\succeq 0, \ Z - x e\tr\hspace{-0.05cm}ex^{\tr} =Z{-}kxx\tr =Z{-}\tfrac{1}{k}{\rm diag}(Z){\rm diag }(Z)\tr \succeq 0.
   \end{align*}
The     following result follows from  the above discussion.

\begin{theorem} \label{thm:SymColor1}
The matrix lifting SDP relaxation~\eqref{pr:sdp3} is equivalent to the following relaxation:
\begin{eqnarray}
\theta_k^3(G) = & \max\limits_{Z\in \Sym^{n}} &  \ip{I}{Z}   \label{pr:sdpSymColor1}   \\
   & \st &  Z_{ij} = 0 \text{ for } \lrc{ij}\in E  \nonumber \\
   && \ Z_{ii} \le 1 \text{ for } i\in [n]  \nonumber \\
   && \ Z\ge 0 \nonumber \\
   && \ \matr{k& {\rm diag}(Z)\tr \\
   {\rm diag}(Z) & Z}\succeq 0,  \nonumber
\end{eqnarray}
and the latter problem is strictly feasible.
\end{theorem}

\begin{proof}
The first part follows from the construction.
To show strict feasibility, consider $M\succ 0$ from  Lemma~\ref{lem:SymColor}.
Let $A_{\bar{G}}$ be the adjacency matrix of the complement of $G$. Then there exists $\eps>0$ such that $M+\eps \tiny{\matr{0&0\tr \\
0&A_{\bar{G}}}} \succ 0$. Therefore matrix $Z=\tfrac{1}{(n+1)}I+\eps A_{\bar{G}}$ is a strictly feasible solution of  \eqref{pr:sdpSymColor1} by construction.
\end{proof}

\medskip
Let us relate our matrix and vector lifting relaxations for the M$k$CS problem. Note that the matrix lifting relaxation does not impose orthogonality constraints that
correspond to the incidence vectors of different colors.
Further, note that ${\rm diag}(Z)\leq e$ is redundant for \eqref{pr:sdpSymColor2} when the orthogonality constraints are imposed, see Lemma \ref{diagredundant}.
Now, when we put those observations together, we arrive to the following result.

\begin{theorem} \label{thm:equivSymColor}
The SDP relaxation \eqref{pr:sdpSymColor1} is equivalent to the following vector lifting relaxation
\begin{align}
 \max_{Z,X\in \Sym^{n}} & \ \ip{I}{Z} \label{pr:sdpSymColorThm} \\
 \st \hspace{0.3cm}   & \ Z_{ij} = 0, \ \text{ for } \lrc{ij}\in E \nonumber \\
   & \ Z_{ii} \le 1,\  \text{ for } i\in [n]  \nonumber \\
  & \ Z\ge 0, \ X\ge 0  \nonumber\\
  & \ Z-X\succeq 0  \nonumber \\
 & \ \matr{1& {\rm diag}(Z)\tr  \\
{\rm diag}(Z) & Z+(k-1)X }\succeq 0.  \nonumber
\end{align} 
\end{theorem}
\begin{proof} First, let $(Z,X)$ be a feasible solution for the SDP relaxation from the theorem.
We claim that $Z$ is feasible for problem~\eqref{pr:sdpSymColor1}, and the corresponding objective values are equal.
The linear constraints in~\eqref{pr:sdpSymColor1} are  readily satisfied.
To verify that the SDP constraint in~\eqref{pr:sdpSymColor1} is also satisfied, we proceed as follows
\[
Z-\tfrac{1}{k}\text{diag}(Z)\text{diag}(Z)\tr {\succeq}  Z-\tfrac{1}{k}Z-\tfrac{k-1}{k}X = \tfrac{k-1}{k}\big( Z-X \big ) {\succeq} 0.
\]
Here we exploit two positive semidefinite constraints from \eqref{pr:sdpSymColorThm}.

Now, let $Z$ be a feasible solution to~\eqref{pr:sdpSymColor1}.
We claim that $(Z,X) := (Z,\tfrac{1}{k}\text{diag}(Z)\text{diag}(Z)\tr)$ is feasible for the SDP relaxation \eqref{pr:sdpSymColorThm},
and the corresponding  objective values are equal. The linear constraints of the  SDP relaxation \eqref{pr:sdpSymColorThm}
are clearly satisfied for previously defined $(Z,X)$.
The SDP constraint of \eqref{pr:sdpSymColor1} implies $Z\succeq 0$  and
\[
Z-X=Z-\tfrac{1}{k}\text{diag}(Z)\text{diag}(Z)\tr\succeq 0,
\]
by the Schur complement. Therefore the first SDP constraint in \eqref{pr:sdpSymColorThm}   is satisfied. Finally,
\[
Z+(k-1)X=Z+\tfrac{(k-1)}{k}\text{diag}(Z)\text{diag}(Z)\tr \succeq 0
\]
and
\[
Z +(k-1)X -\text{diag}(Z)\text{diag}(Z)\tr=Z+(k-1)X-kX= Z-X\succeq 0,\]
which implies the second SDP constraint in \eqref{pr:sdpSymColorThm}.
\end{proof}

Note that the feasible region of the SDP relaxation from Theorem \ref{thm:equivSymColor} is a superset of the feasible region of the SDP relaxation for $\theta_k^2(G)$.
The following result follows directly from the previous theorem and Theorem \ref{thm:equiv0}:
\[
\theta_k^1(G) \leq  \theta_k^2(G) \leq \theta_k^3(G),
\]
where $\theta_k^1(G)$ is the optimal solution of the SDP relaxation \eqref{pr:sdpSymColor2},
$\theta_k^2(G)$ is the optimal solution of  \eqref{pr:sdpSymColor2} without constraints \eqref{ineq:symColor3},~\eqref{ineq:symColor4},
and $\theta_k^3(G)$ is the optimal solution of \eqref{pr:sdpSymColor1}.

\medskip
In the previous section we concluded that $\theta^3_k(G)=\vartheta'(G)$ when  $k=1$, where $\vartheta'(G)$ is the  Schrijver number.
We were not able to establish a relation between $\theta^3_k(G)$ and $\vartheta'(G)$ when $k>1$.
However, our numerical results in Section~\ref{sec:num} suggest the following result.
\begin{conjecture}\label{conj:thetaLarger}
For $k\geq 2$, the upper bound $\theta^3_k(G)$, see \eqref{pr:sdpSymColor1}, is at least as good as the upper bound $\vartheta'_k(G)$, see \eqref{pr:theta'}.
\end{conjecture}

\medskip

We conclude this section listing some cases when our bounds are tight.

\begin{lemma} \label{lem:tight} For a given $k$, let $G$ be a graph such that $\alpha_k(G)=k\vartheta(G)$, then
\begin{equation*}\label{eq:tight}
\vartheta_k(G)=\vartheta'_k(G)=\theta^1_k(G)=\theta^2_k(G)=\theta^3_k(G)=k\alpha(G)=\alpha_k(G).
\end{equation*}
\end{lemma}
\begin{proof}
Let $G$ be any graph.  We have $\alpha_k(G) \le k\alpha(G) \le k\vartheta(G)$.
We have already shown  $\alpha_k(G) \le \theta_k^1(G) \le  \theta_k^2(G) \le \theta_k^3(G)$.
We claim that $\theta^3_k(G) \le k\vartheta(G)$.
To prove this, given $Z$ a feasible solution to~\eqref{pr:sdpSymColor1} let  $\hat Z = \frac 1{\ip{I}{Z}} Z$.  Then $\hat Z$ is feasible for~\eqref{pr:thetaLov2}. It is enough to show then that $k\ip{J}{\hat Z} \ge {\ip{I}{Z}}$ that is, $ k\ip{J}{Z} \ge {\ip{I}{Z}}^2$. But, using the Schur complement, the last constraint from~\eqref{pr:sdpSymColor1} is equivalent to $kZ - {\rm diag }(Z){\rm diag } (Z)\tr \succeq 0$ which implies
$ k\ip{J}{Z} - {\ip{I}{Z}}^2 = ke\tr Ze - (e\tr{\rm diag } Z)^2 = e\tr(kZ - {\rm diag } (Z){\rm diag } (Z)\tr)e  \ge 0$.

Also, $\alpha_k(G)\le \vartheta'_k(G) \le \vartheta_k(G)$.
We claim that $ \vartheta_k(G) \le k\vartheta(G)$. To prove this notice that if $Z$ is a feasible solution to~\eqref{pr:theta4} then $\frac 1k Z$ is feasible for~\eqref{pr:thetaLov2}.

Now, if $G$ is such that $\alpha_k(G)=k\vartheta(G)$, all previous inequalities become equalities.
\end{proof}

Notice that $\alpha_k(G) \le k\alpha(G) \le k\vartheta(G)$ and thus the assumption $\alpha_k(G)=k\vartheta(G)$ in Lemma~\ref{lem:tight} is equivalent to $\alpha_k(G) = k\alpha(G)$ and $\alpha(G) = \vartheta(G)$.
Several families of graphs satisfy these conditions. For instance if $G$ is a perfect graph with at least $K$ non-intersecting independent sets of size $\alpha(G)$, then  $\alpha_k(G)= k\alpha(G) =k\vartheta(G)$ for all $k \le K$.  In Proposition~\ref{prop:tight}, we characterize a family of (not necessarily perfect) graphs satisfying the conditions of
Lemma~\ref{lem:tight}. The condition is given in terms of  the chromatic number of  the complement graph of $G$, $\chi(\bar G)$. Notice that $\chi(\bar G)$ is equal to the clique cover number of $G$.

\begin{proposition}\label{prop:tight} Let $G$ be a graph such that $|V(G)| = \chi(\bar G)\chi(G)$. Then Lemma \ref{eq:tight} holds for all $k \le \chi(G)$.
\end{proposition}
\begin{proof}
For any graph $\chi(\bar G) \ge \alpha(G) \ge |V(G)|/\chi(G)$. Thus $|V(G)| = \chi(\bar G)\chi(G)$ implies $\chi(\bar G) = \alpha(G)$ and $\alpha_{\chi(G)}(G) = \chi(G)\alpha(G)$.
Using the Lov\'asz sandwich theorem \cite{Lovasz}
  $\alpha(G) \le \vartheta(G) \le \chi(\bar G)$
   we obtain that $\alpha(G) = \vartheta(G)$, from where it follows the statement.
\end{proof}

The set of vertex-transitive graphs contains a number of non-trivial examples for Proposition~\ref{prop:tight}, for instance see Table~\ref{tab:tight}.
A graph is vertex-transitive if its automorphism group acts transitively on vertices,
that is, if for every two vertices there is an automorphism that maps one to the other.
In the table $H(v,d):=H(v,d,1)$ denotes the Hamming graph, while $J(v,d):=J(v,d,d-1)$ denotes the Johnson graph.
For definitions of this graphs see Section~\ref{sec:GraphSym}.
\begin{table}[H]
\caption{Graphs for which all our SDP upper bounds are tight.}
\begin{center}
\renewcommand{\tabcolsep}{4.5pt}
\renewcommand{\arraystretch}{1.1}
\begin{tabular}{|l|ccc|c|}
\hline
{Graph} & $|V(G)|$ & $\chi(G)$ & $\chi(\bar G)$ &  \\
\hline
$H(v,d)$ & $d^{v} $ & $d$ & $d^{v-1}$ &\cite{ProductGraphs}\\
$J(v,2)$, $v$   even & ${v \choose 2}$ & $v-1$ & $\tfrac v2$ &\cite{ColorJohnson}\\
$J(v,3)$, $v \equiv 1 \text{ or } 3 \text{ mod } 6$  &  ${v \choose 3}$  &  $v-2$  & $\tfrac{v(v-1)}{6}$ &\cite{AlphaJohnson},\cite{ColorJohnson} \\ \hline
\end{tabular}
\end{center}
\label{tab:tight}
\end{table}

\section{Reductions using graph symmetry} \label{sec:GraphSym}

In this section we first prove that several inequalities in the strongest vector lifting relaxation \eqref{pr:sdpSymColor2} are redundant  for vertex-transitive graphs.
Then, we present reduced  SDP relaxations for different classes of highly symmetric graphs.
We say that a graph is highly symmetric if its adjacency matrix belongs to an association scheme, see e.g., \cite{Delsarte,SchrijverTheta,vanDam2015}.

\begin{proposition}
  Let $G$ be a vertex-transitive graph.
  Then constraints \eqref{ineq:symColor4} are redundant for the SDP relaxation  \eqref{pr:sdpSymColor2obj}--\eqref{ineq:symColor2}.
\end{proposition}
\begin{proof}  Let $(Z,X)$ be an optimal solution for the SDP relaxation  \eqref{pr:sdpSymColor2obj}--\eqref{ineq:symColor2}.
By averaging  and vertex transitivity of the graph we obtain an optimal solution $(\bar Z,\bar X)$ such that $\bar Z_{ii} = \bar Z_{jj}$ for all $i,j$.
Given $i$ and $j$, let $z = \bar Z_{ii} =\bar Z_{jj}$ and $y = \bar Z_{ij} + (k-1) \bar X_{ij}$.
Looking at the principal sub-matrix of constraint \eqref{ineq:symColor2}  indexed by $i$ and $j$ we obtain
\[
\matr{ z  & y \\
y & z }\succeq 0.
\]
Which implies  $z^2  \ge y^2$, that is equivalent to  \eqref{ineq:symColor4} from the non-negativity of $z$ and $y$.
\end{proof}

In general, constraints \eqref{ineq:symColor3}  are not redundant for the SDP relaxation
 \eqref{pr:sdpSymColor2} and  vertex-transitive graphs. For example, for  the Petersen graph and $k=2$ we have that  $\theta_k^1(G)$ equals  7.5,
while $\theta_k^2(G)$ is equal to $8$.

Below  we consider symmetry  reduction for the Hamming and Johnson graphs.
We apply the general theory of symmetry reduction to our SDP relaxations, see e.g., \cite{deKlerk2008,vanDam2015,SchrijverTheta,Parillo}, and therefore omit details.
While in the mentioned and other papers in the literature, symmetry reduced relaxations are linear programming relaxations,
our simplified relaxations have also second order cone constraints.

Let us now define the Hamming graphs. The vertex set $V$ is the set of $d$-tuples of letters from an alphabet of size $q$, so $n:=|V|=q^d$.
The adjacency matrices $H(d,q,j)$ ($j=0,\ldots, d$) of the Hamming association scheme are defined by the number of positions in which two $d$-tuples differ.
In particular, $H(d,q,j)_{x,y}=1$ if $x$ and $y$ differ in $j$ positions, for $x,y\in V$ ($j=0,\ldots, d$), i.e., if their Hamming distance $d(x,y)=j$.
$H(d,q,1)$ is the adjacency matrix of the well-known Hamming graph, which can also be obtained as the Cartesian  product of $d$ copies of the complete graph $K_q$.
  Further, we denote by $H^-(d,q,j)$    the graph whose Hamming distance $d(x,y)\leq j$.
The matrices  of the Hamming association scheme can be simultaneously diagonalized.
The eigenvalues (character table) of the Hamming scheme can be expressed in terms of Krawtchouk polynomials:
\[
K_{i}(u) := \sum_{j=0}^{i} (-1)^{j}(q-1)^{i-j}{u \choose j} {d-u \choose i-j}, \quad  i,u = 0,\ldots,d.
\]
In particular, eigenvalues of $B_i := H(d,q,i)$ ($i=0,1,\ldots,d$) are  $K_i(j)$ for $j=0,1,\ldots,d$.

Now, let us consider the SDP relaxation \eqref{pr:theta'}.
Since the relaxation is invariant under the permutation group of the Hamming graph, we can restrict optimization of the SDP relaxation to feasible points in the
Bose-Mesner algebra, see e.g., \cite{Parillo,deKlerk2008}. Therefore, we assume $Z = \sum_{i=0}^d z_i B_i$ in \eqref{pr:theta'}.
Further we consider the case in which the adjacency matrix corresponds to $H(d,q,1)$.
Relaxations are similar for any other $H(d,q,j)$ or  $H^-(d,q,j)$  ($j=2\ldots, d$).
After the substitution,  the SDP relaxation \eqref{pr:theta'} reduces to:
\begin{align}
\vartheta'_k(G) = \max_{z \in \R^{d+1}} & \ k+ \sum_{i=2}^d z_i \ip{J}{B_i} \label{pr:theta'Red} \\
 \st \hspace{0.2cm}  &   \tfrac{k}{n}  + \sum_{i=2}^d z_i K_i(j) \geq 0,  \text{ for } j\in \{0,1,\ldots, d\}   \nonumber\\
 &  1- \tfrac{k}{n}  - \sum_{i=2}^d z_i K_i(j) \geq 0,  \text{ for } j\in \{0,1,\ldots, d\}   \nonumber\\
 & z_0=\tfrac{k}{n}, ~z_1=0,~ z_i\geq 0,   \text{ for } i\in \{2,\ldots, d\}.  \nonumber
\end{align}
Note that \eqref{pr:theta'Red} is a linear program.
Next, we reduce our matrix lifting  SDP relaxation \eqref{pr:sdpSymColor1} by using similar arguments as before.
The resulting $\theta^3$-bound is as follows:
\begin{align}
\theta^3(G) = \max_{z \in \R^{d+1}} & \ n \cdot z_0 \label{pr:thetaMatr'Red} \\
 \st \hspace{0.2cm}  &   \sum_{i=0}^d z_i K_i(0) - \tfrac{n}{k} z_0^2 \geq 0  \nonumber\\
 &  \sum_{i=0}^d z_i K_i(j) \geq 0,  \text{ for } j\in \{1,\ldots, d\}   \nonumber\\
 & z_0\leq 1, ~z_1=0, ~z_i\geq 0,    \text{ for } i\in \{0,2,3,\ldots, d\}.  \nonumber
\end{align}
Note that \eqref{pr:thetaMatr'Red} is an optimization problem with a linear objective, $2d+1$ linear inequalities and one convex quadratic constraint.
Finally,  we simplify the SDP relaxation whose optimal value is denoted by $\theta^2(G)$, i.e.,
the vector lifting SDP relaxation \eqref{pr:sdpSymColor2obj} without \eqref{ineq:symColor3} and \eqref{ineq:symColor4}.
Here, we also may restrict  $X = \sum_{i=0}^d x_i B_i$.
\begin{align}
\theta^2(G) = \max_{z,x \in \R^{d+1}} & \ n \cdot z_0 \label{pr:thetaVec'Red} \\
 \st \hspace{0.2cm}  &  \sum_{i=0}^d (z_i-x_i)K_i(j)\geq 0,   \text{ for } j\in \{0,1,\ldots, d\} \nonumber\\
 & \sum_{i=0}^d z_i K_i(0) + (k-1)\sum_{i=0}^d x_i K_i(0)  - n z_0^2 \geq 0  \nonumber\\
 &  \sum_{i=0}^d z_i K_i(j) + (k-1)\sum_{i=0}^d x_i K_i(j)  \geq 0,  \text{ for } j\in \{1,\ldots, d\}   \nonumber\\
 & x_0=0, ~x_i\geq 0,    \text{ for } i \in \{1,2,3,\ldots, d\}  \nonumber \\
  & z_0\leq 1, ~z_1=0, ~z_i\geq 0 ,   \text{ for } i \in \{0,2,3,\ldots, d\}.  \nonumber
\end{align}
Note that the optimization problem \eqref{pr:thetaVec'Red} has a linear objective, one second order cone constraint, and several linear constraints. \\

Finally, to compute $\theta^1(G)$, we add to \eqref{pr:thetaVec'Red} the symmetry reduced inequalities \eqref{ineq:symColor3}, i.e.,
\[
1 - 2z_0 + z_i+(k-1)x_i \geq 0, \qquad i=1, \ldots, d.
\]

One can similarly derive simplified SDP relaxations for graphs whose corresponding algebra is diagonalizable, such as for the Johnson graph  $J(v,d,q)$.
The Johnson graph is defined as follows. Let $\Omega$ be a fixed set of size $v$ and let $d$ be an integer such that $1\leq d\leq v/2$.
The vertices of the Johnson graph  $J(v,d,q)$ are the subsets of $\Omega$ with size $d$. Two vertices are connected if the corresponding sets have $q$ elements in common.
In the literature, the  graph $J(v,d,d-1)$  is known as the Johnson graph $J(v,d)$, while  $J(v,d,0)$ is known as the Kneser graph  $K(v,d)$.
Matrices corresponding to $J(v,d,q)$, $q=0,1,\ldots, d$ can be simultaneously diagonalized.
The eigenvalues (character table) of the Johnson scheme can be expressed in terms of Eberlein polynomials:
\[
E_{i}(u) := \sum_{j=0}^{i} (-1)^{j} {u \choose j} {d-u \choose i-j}{v-d-u \choose i-j}, \quad  i, u = 0,\ldots,d.
\]
Eigenvalues of $J(v,d,i)$ ($i=0,1,\ldots,d$) are  $E_i(j)$ for $j=0,1,\ldots,d$.
Now, one can proceed similarly as with the Hamming graphs  in order to obtain simplified relaxations for the Johnson graphs.
The resulting relaxations for the Johnson graphs differ from the relaxations for the Hamming graphs  in the type of polynomials.

\section{Symmetry reductions for other partition problems} \label{subsec:partit}

Notice that a $k$-colorable subgraph of a graph corresponds to a partition of the graph's vertices into $k{+}1$ subsets, i.e., $k$ independent sets and the rest of the vertices.
Therefore one can consider the $k$-colorable subgraph problem as a graph partition problem, which is  invariant under permutations of the subsets.
It is not difficult to verify that other graph partition problems such as the max-$k$-cut problem and the $k$-equipartition problem  are also invariant under permutations of the subsets.
The max-$k$-cut problem is the problem of partitioning the vertex set of a graph into  $k$ sets such that the total weight
of edges joining different sets is maximized. For the problem formulation and related  SDP relaxations see e.g., \cite{Delorme1993,deKlerk2004,Rendl2012}.
The $k$-equipartition problem is the problem of partitioning the vertex set of a graph into $k$ sets
of equal cardinality  such that the total weight of edges joining different sets  is minimized.
For the problem formulation and related  SDP relaxations see e.g., \cite{KarishEquipart,WolkZhao,HandbookSDP,vanDam2015}.

It is known that  vector and matrix lifting SDP relaxations  for the max-$k$-cut and $k$-equipartition problems are equivalent.
In particular, De Klerk et al.~\cite{deKlerk2004}  prove the equivalence of the relaxations for the max-$k$-cut problem,
by exploiting the invariance of the max-$k$-cut problem under permutations of the subsets.
\citet{HandbookSDP}  proves the equivalence of three different SDP relaxations for the $k$-equipartition problem;
a matrix lifting relaxation,  a vector lifting relaxation,  and an SDP relaxation for the  $k$-equipartition problem derived
as a special case of the quadratic assignment problem.

Here, we  prove the same results  by using the approach from Section \ref{sec:SymColorV} and  \ref{sec:SymColorM}.
We remark that the proof here is more elegant than the one from \cite{HandbookSDP}.

We denote the optimal value of the vector lifting SDP relaxation for the max-$k$-cut (resp., $k$-equipartition) problem on graph $G$ by $MkC_v(G)$  (resp., $Ek_v(G)$).
The vector lifting relaxations of both problems are particular cases of the relaxation for the general graph partition problem by~\citet{WolkZhao}.
Let $L$ be the Laplacian matrix of $G$, then the symmetry-reduced versions of the vector lifting relaxations are given as follows:

\hspace{0.0cm}\begin{minipage}{.25\textwidth}
\begin{align}
MkC_v(G) = \max_{Z,X\in \Sym^{n}} & \ \tfrac{1}{2}\ip{L}{Z} \nonumber  \\
 \st \hspace{0.3cm}   &  \ X_{ii} = 0, \ \text{ for } i\in [n]  \nonumber \\
  & \ Z_{ii} = 1,\  \text{ for } i\in [n]  \nonumber \\
  & \ Z\ge 0, \ X\ge 0  \nonumber \\
  & \ Z-X\succeq 0 \nonumber \\
 & \ Z+(k-1)X -J \succeq 0.  \nonumber \\
 &\nonumber
\end{align}
\end{minipage}
\begin{minipage}{.35\textwidth}
\begin{align}
Ek_v(G) = \min_{Z,X\in \Sym^{n}} & \ \tfrac{1}{2}\ip{L}{Z} \nonumber \\
 \st \hspace{0.3cm}   &  \ X_{ii} = 0, \ \text{ for } i\in [n]  \nonumber \\
  & \ Z_{ii} = 1,\  \text{ for } i\in [n]  \nonumber \\
  & \ Z\ge 0, \ X\ge 0  \nonumber \\
  & \ Z-X\succeq 0 \nonumber \\
 & \ Z+(k-1)X -J \succeq 0 \nonumber\\
 & \ Z e= \tfrac{n}{k}e. \nonumber
\end{align}
\end{minipage}\vspace{0.5cm}

Next, we look at the matrix lifting relaxations for the two problems.
For the max-$k$-cut problem, reducing the matrix lifting SDP relaxation results in the relaxation by \citet{kCut3}, which is equivalent to the relaxation by \citet{kCut1}.
Similarly, for the $k$-equipartition problem, reducing the  matrix lifting SDP relaxation results in a well-known SDP relaxation by \citet{KarishEquipart},  which is equivalent to the relaxation by~\citet{kCut2}.
Both relaxations are presented below, and  notation is analogous to the notation in the vector lifting case.

\hspace{0cm}\begin{minipage}{.3\textwidth}
\begin{align}
MkC_m(G) = \max_{Z\in \Sym^{n}} & \ \tfrac{1}{2}\ip{L}{Z} \nonumber \\
 \st \hspace{0.3cm}   &   \ Z_{ii} = 1,\  \text{ for } i\in [n] \hspace{0.9cm} \nonumber \\
  & \ Z\ge 0,  \nonumber \\
 & \ Z-\tfrac{1}{k} J\succeq 0. \nonumber\\
 &\nonumber
\end{align}
\end{minipage}
\begin{minipage}{.35\textwidth}
\begin{align}
Ek_m(G) = \min_{Z\in \Sym^{n}} & \ \tfrac{1}{2}\ip{L}{Z} \nonumber \\
 \st \hspace{0.3cm}   &  \  Z_{ii} = 1,\  \text{ for } i\in [n]  \nonumber \\
  & \ Z\ge 0 \nonumber \\
 & \ Z-\tfrac{1}{k} J\succeq 0 \nonumber\\
 & \ Z e= \tfrac{n}{k}e. \nonumber
\end{align}
\end{minipage}\vspace{0.5cm}

To show that the vector and matrix lifting relaxations are equivalent, observe that for both the $k$-equipartition and max-$k$-cut problem one can construct feasible solutions to the matrix lifting relaxation from the vector lifting relaxation with the same objective value, and the other way round. First, from a feasible solution $Z$ for each of the symmetry-reduced matrix lifting relaxations we obtain
a feasible solution $(Z,\tfrac{1}{k-1}(J-Z))$ for the corresponding symmetry-reduced vector lifting relaxation.
For the opposite direction, a feasible solution $(Z,X)$ to each of the symmetry-reduced vector lifting relaxations provides a feasible solution $Z$ to the corresponding matrix lifting relaxation,
as in the case of the M$k$CS problem, see also the proof of Theorem \ref{thm:equivSymColor}.
Hence the vector and matrix lifting relaxations are equivalent.

\section{Numerical results} \label{sec:num}

In Section \ref{subsec:boundAlphaK} we compare our upper and lower bounds with the bounds for the M$k$CS problem on  graphs  from the literature.
In  Section \ref{subsec:symmgraphsNumeric} we present  upper and lower bounds  for highly  symmetric graphs of larger sizes by exploiting the results from Section~\ref{sec:GraphSym}.
In Section \ref{subsec:boundChi}, we exploit our bounds for the  M$k$CS problem to compute bounds on the  chromatic number $\chi(G)$ of a graph $G$.

All computations are done in MATLAB R2018b with Yalmip~\cite{Lofberg2004} on a computer with two processors Intel\textsuperscript{\textregistered} Xeon\textsuperscript{\textregistered}
Gold 6126 CPU @ 2.60 GHz and 512 GB of RAM. Semidefinite programs are solved with MOSEK, Version 8.0.0.80.\\

\noindent{\bf Lower bounds.}
To obtain lower bounds for the M$k$CS problem, we use two  heuristics.
Our first heuristic is based on the MATLAB heuristic algorithm for finding  maximum independent sets~\cite{heuristic}.
Namely, we transform the M$k$CS to the stable set problem as described   in    Section~\ref{sec:alpha}, and then use the mentioned algorithm.
 The heuristic first finds the vertices with the minimum degree. If there is only one such vertex, it is added to the stable set, and its neighbors are excluded from the stable set. If there are several vertices with the minimum degree, then the heuristic looks at the support of each such vertex, i.e., the sum of degrees of this vertex's neighbors. The vertex with the largest support is added to the  stable set, and its neighbors are excluded from it. If there are several vertices with the largest support, they are chosen according to the priority specified by the user in advance. The procedure is repeated until all vertices are considered.

 Our second heuristic is a tabu search algorithm.
  A greedy heuristic is used to find the initial solution. Our greedy algorithm sorts vertices of a graph in  order of ascending degree,
and going down this list, tries to color as many vertices as possible with the same color. When the end of the list is reached,
the colored nodes are removed from the list and the algorithm  starts again at the top of the list with the next color.
After the initial solution is found, a tabu list is used to obtain a better solution.
In each iteration of the main algorithm one of the colored vertices that has the most edges in common with non-colored vertices is uncolored,
and the other uncolored nodes are checked if they can be colored. For those that get a color, we prevent them from being uncolored for a certain number of iterations.
The tabu search algorithm outperforms the other heuristic on dense graphs.

\subsection{Bounds for the M$k$CS problem for general graphs} \label{subsec:boundAlphaK}

In this section, we compute upper bounds on $\alpha_k(G)$ for benchmark graphs from the literature.
In particular, we compare our bounds with the bounds from \citet{benchmark1,JanuschowskiBCP}.
In the former paper, the authors present an IP formulation of the M$k$CS and implement a parallel subgradient algorithm.
In the latter paper, the authors propose a branch-and-cut method that accounts for both,
the symmetry with respect to  color permutations and the inner graph symmetry.\\

\noindent
{\bf Graphs.} Graphs used in the mentioned two papers are from two different sources.
We first describe graphs from the Second DIMACS Implementation Challenge for the max-clique problem~\cite{dimacs}.
Graphs ``brock$x$\_$y$"  are random graphs with $x$ vertices and depth of clique hiding $y$,   see \cite{BrockPaper}.
Graphs ``gen$x$\_p$y$-$z$"  are artificially generated with $x$ vertices, edge density $y$ and known embedded clique of the size  $z$.
``san$x$\_p$y$\_$z$" and ``sanr$x$\_p$y$ are randomly generated graph instances where parameters $x$, $y$ and $z$ have the same meaning as for ``gen" instances.
 ``C$x$.$y$" are random graphs on $x$ vertices with edge probability $y$.
 ``p-hat" graphs are generated with the $p$-hat generator, which generalizes the classical uniform random graph generator and
  the resulting graphs have larger cliques than uniform graphs,  see~\cite{Gendreau1993}.
Graphs ``keller$x$" are based on Keller's conjecture on tilings using hypercubes.
 ``c-fat$x$-$y$" are graphs  with $x$ vertices based on fault diagnosis problems.

The second group of graphs are from the COLOR02 symposium~\cite{color02}.
``Myciel" are graphs based on the Mycielski transformation.
They are triangle free, but the coloring number increases in problem size so that the graphs can have arbitrary large gaps between their clique and chromatic numbers.
The ``FullIns" and ``Insertions" graphs are  generalizations of the Mycielski graphs.
``DSJC$n$.$p$" are  standard random graphs with $n$ vertices where an edge between two vertices appears with probability $p$.
``DSJR$n$.$p$" are graphs with $n$ vertices randomly distributed in the unit square with an edge between two vertices if the length of the line between them is less than $p/10$.
 ``queen$n$-$n$" graph is a graph with vertices that correspond to the squares of the $n{\times} n$ chess board  and are connected by an
edge if the corresponding squares are in the same row, column, or diagonals (according to the queen move rule at the chess game).
``MANN\_a$x$" is the graph from the clique formulation of the set covering problem for the Steiner triple system on a set with $x$ elements see, e.g.,~\cite{MANNINO1995} for the problem description.

Graphs in Tables~\ref{tab:1Color} to \ref{tab:4ColorTimes} that are marked by a superscript $^{02}$ were used as benchmarks in the COLOR02 symposium. Other graphs are benchmarks from the Second DIMACS Challenge~\cite{dimacs} or their complements\footnote{The description in \cite{benchmark1} does not mention taking the complements of the graphs, but the edge densities and  bounds in the paper correspond to the complements of the DIMACS graphs.}. All graphs in the tables that end with a superscript $^c$ are complements of the original graphs. \\

Table~\ref{tab:1Color} (resp., Table \ref{tab:2Color}) compares our bounds with the results from \citet{benchmark1} (resp., \citet{JanuschowskiBCP})  for graphs up to 200 vertices.
The first column in Table~\ref{tab:1Color}  lists graphs, while the second, third, and fourth columns list the number of vertices, number of edges,  and density of the graphs, respectively.
 The fifth column in Table~\ref{tab:1Color}  specifies the number of colors $k$, and sixth  upper bounds on $\alpha_k(G)$ from  \citet{benchmark1}.
 In the remaining columns we list the  following upper bounds:
$\vartheta_k(G)$ see \eqref{pr:theta4},  $\vartheta'_k(G)$ see  \eqref{pr:theta'},
 $\theta^3_k(G)$ see \eqref{pr:sdpSymColor1},
$\theta^2_k(G)$ that is the optimal solution of  \eqref{pr:sdpSymColor2} without constraints \eqref{ineq:symColor3},~\eqref{ineq:symColor4},
and  $\theta^1_k(G)$ see \eqref{pr:sdpSymColor2}.
We also compute $\theta_k^1(G)$ with BQP inequalities~\eqref{ineq:31} and~\eqref{ineq:32}.
In particular, we use a cutting plane method that  adds up to $2n$ most violated  BQP inequalities for at  most four iterations.
In the last column of Table~\ref{tab:1Color} we list lower bounds obtained by one of our  heuristics. We report only the best lower bound.
Table~\ref{tab:2Color} is organized similarly to Table~\ref{tab:1Color}.
However, the sixth column of Table \ref{tab:2Color} lists the optimal solution for $\alpha_k(G)$ computed by \citet{JanuschowskiBCP}.

We highlight by boldface the best upper bounds for a given graph and $k$. All upper bounds are rounded to the nearest second digit.
We omit in this section the results for the Johnson and Hamming graphs used in  \cite{benchmark1,JanuschowskiBCP} since we devote the whole next section to highly symmetric graphs.
Table \ref{tab:1ColorTimes} and \ref{tab:2ColorTimes} report MOSEK running times in seconds.

Our numerical results in Table~\ref{tab:1Color}  show that for all compared graphs, except ``c-fat200-5$^c~$", our upper bounds dominate  the upper bounds from~\cite{benchmark1}.
In particular, for eight out of ten graphs and all tested $k$ our best SDP bounds improve upon the upper bounds from~\cite{benchmark1},
and for ``c-fat200-2$^c~$" already $\theta_k^3$ bound is tight for both $k$.
Our lower bounds significantly improve the lower bounds from \cite{benchmark1} for most of the instances.
 Table \ref{tab:2Color} shows that our upper bounds, after rounding  down to the nearest integer,  provide  optimal values for nine  out of twenty one  instances.

Our best upper bound, as expected, is $\theta_k^1(G)$ with BQP inequalities~\eqref{ineq:31} and~\eqref{ineq:32}.
However, in many cases several  SDP relaxations provide the same bound as the strongest one,
while the  computational  times for the vector lifting relaxations are substantially larger than the  computational  times for other relaxations.
Table \ref{tab:1ColorTimes} and \ref{tab:2ColorTimes} also show that it takes less computational time to solve SDP relaxations  for dense graphs.
Note that several of the graphs considered here have additional symmetry (see \citet{JanuschowskiBCP} for details)
that can be exploited to reduce computational effort. However,  we do not exploit symmetry for those graphs here.
Since the running times of the SDP models and those from  \cite{benchmark1} are not comparable due to the use of
different machines, we do not present the latter ones.

For all graphs the bound $\theta^3_k(G)$  is at least as good as the bound $\vartheta'_k(G)$ that dominates  $\vartheta_k(G)$.
However, computational times for all three bounds are similar.

\begin{table}[H]
\caption{\small{{Bounds} for  graphs with up to 200 vertices considered by \citet{benchmark1}.}}
\begin{center}
\renewcommand{\tabcolsep}{4.5pt}
\renewcommand{\arraystretch}{1.6}
{\scriptsize{
\begin{tabular}{|l|c|c|c|c|c|c|c|c|c|c|c|c|}
\hline
Graph  & $n$ & $|E|$ &$\rho$,\%  & $k$ &  UB (C$\&$C)  & $\vartheta_k$ & $\vartheta_k'$ & $\theta^3_k$ & $\theta^2_k$ & $\theta^1_k$ & \begin{tabular}[c]{@{}c@{}}$\theta^1_k$\\ $+$ BQP \end{tabular}   & \begin{tabular}[c]{@{}c@{}} LB  \end{tabular} \\ \hline

{brock200\_2$^c$} & 200 & 10024 & 50 & 2 & 30.10 & 28.45 & \textbf{28.26} & \textbf{28.26} & \textbf{28.26} & \textbf{28.26} & \textbf{28.26} & 19 \\
 &  &  &  & 3 & 44.90 & 42.68 & \textbf{42.39} & \textbf{42.39} & \textbf{42.39} & \textbf{42.39} & \textbf{42.39} & 28 \\ \hline
{brock200\_4$^c$} & 200 & 6811 & 34 & 2 & 51.80 & 42.59 & \textbf{42.24} & \textbf{42.24} & \textbf{42.24} & \textbf{42.24} & \textbf{42.24} & 30 \\
 &  &  &  & 3 & 76.60 & 63.88 & \textbf{63.36} & \textbf{63.36} & \textbf{63.36} & \textbf{63.36} & \textbf{63.36} & 42 \\ \hline

\multicolumn{ 1}{|l|}{C125.9$^c$} & \multicolumn{ 1}{c|}{125} & \multicolumn{ 1}{c|}{787} & \multicolumn{ 1}{c|}{10} & 2 & 79.40 & 75.61 & 75.09 & 74.63 & 74.41 & 74.11 & \textbf{74.10} & 64 \\ 
\multicolumn{ 1}{|l|}{} & \multicolumn{ 1}{c|}{} & \multicolumn{ 1}{c|}{} & \multicolumn{ 1}{c|}{} & 3 & 115.60 & 112.86 & 112.18 & 107.27 & 106.96 & 105.90 & \textbf{105.31} & 89 \\ \hline

\multicolumn{ 1}{|l|}{keller4$^c$} & \multicolumn{ 1}{c|}{171} & \multicolumn{ 1}{c|}{5100} & \multicolumn{ 1}{c|}{35} & 2 & 27.90 & 28.02 & \textbf{26.93} & \textbf{26.93} & \textbf{26.93} & \textbf{26.93} & \textbf{26.93}
 & 22 \\
\multicolumn{ 1}{|l|}{} & \multicolumn{ 1}{c|}{} & \multicolumn{ 1}{c|}{} & \multicolumn{ 1}{c|}{} & 3 & 41.90 & 42.04 & \textbf{40.40} & \textbf{40.40} & \textbf{40.40} & \textbf{40.40} & \textbf{40.40} & 31 \\ \hline

\multicolumn{ 1}{|l|}{{c-fat200-2$^c$}} & \multicolumn{ 1}{c|}{200} & \multicolumn{ 1}{c|}{16665} & \multicolumn{ 1}{c|}{84} & 2 & \textbf{46.00} & 46.33 & 46.33 & \textbf{46.00} & \textbf{46.00} & \textbf{46.00} & \textbf{46.00} & 46 \\
\multicolumn{ 1}{|l|}{} & \multicolumn{ 1}{c|}{} & \multicolumn{ 1}{c|}{} & \multicolumn{ 1}{c|}{} & 3 & \textbf{68.00} & 68.33 & 68.33 & \textbf{68.00} & \textbf{68.00} & \textbf{68.00} & \textbf{68.00} & 68 \\ \hline
\multicolumn{ 1}{|l|}{c-fat200-5$^c$} & \multicolumn{ 1}{c|}{200} & \multicolumn{ 1}{c|}{11427} & \multicolumn{ 1}{c|}{57} & 2 & \textbf{116.00} & 120.69 & 120.69 & 120.69 & 120.69 & 120.69 & 119.42 & 116 \\
\multicolumn{ 1}{|l|}{} & \multicolumn{ 1}{c|}{} & \multicolumn{ 1}{c|}{} & \multicolumn{ 1}{c|}{} & 3 & \textbf{172.00} & 181.04 & 181.04 & 181.03 & 181.01 & 175.99 & 175.24 & 172 \\ \hline
\multicolumn{ 1}{|l|}{gen200\_p0.9\_44$^c$} & \multicolumn{ 1}{c|}{200} & \multicolumn{ 1}{c|}{1990} & \multicolumn{ 1}{c|}{10} & 2 & 88.00 & 88.00 & 88.00 & 88.00 & 88.00 & 88.00 & \textbf{87.99} & 81 \\
\multicolumn{ 1}{|l|}{} & \multicolumn{ 1}{c|}{} & \multicolumn{ 1}{c|}{} & \multicolumn{ 1}{c|}{} & 3 & 132.00 & 132.00 & 132.00 & 131.94 & 131.90 & 131.84 & \textbf{131.66} & 114 \\ \hline

\multicolumn{ 1}{|l|}{gen200\_p0.9\_55$^c$} & \multicolumn{ 1}{r|}{200} & \multicolumn{ 1}{c|}{1990} & \multicolumn{ 1}{c|}{10} & 2 & 109.00 & 103.23 & 102.79 & 100.84 & 100.52 & \textbf{100.36} & \textbf{100.36} & 93 \\
\multicolumn{ 1}{|l|}{} & \multicolumn{ 1}{r|}{} & \multicolumn{ 1}{r|}{} & \multicolumn{ 1}{c|}{} & 3 & 161.60 & 150.52 & 149.70 & 146.22 & 146.01 & 145.61 & \textbf{145.34} & 128 \\ \hline

\multicolumn{ 1}{|l|}{san200\_0.7\_2$^c$} & \multicolumn{ 1}{c|}{200} & \multicolumn{ 1}{c|}{5970} & \multicolumn{ 1}{c|}{30} & 2 & 36.00 & 36.00 & 35.68 & 35.63 & 35.62 & 35.60 & \textbf{35.60} & 31  \\
\multicolumn{ 1}{|l|}{} & \multicolumn{ 1}{c|}{} & \multicolumn{ 1}{c|}{} & \multicolumn{ 1}{c|}{} & 3 & 54.00 & 53.98 & 53.34 & 53.24 & 53.24 & 53.24 & \textbf{53.22} & 44 \\ \hline
\multicolumn{ 1}{|l|}{san200\_0.9\_2$^c$} & \multicolumn{ 1}{c|}{200} & \multicolumn{ 1}{c|}{1990} & \multicolumn{ 1}{c|}{10} & 2 & 117.00 & 109.05 & 108.68 & 106.61 & 106.24 & \textbf{106.03} & \textbf{106.03} & 99 \\
\multicolumn{ 1}{|l|}{} & \multicolumn{ 1}{c|}{} & \multicolumn{ 1}{c|}{} & \multicolumn{ 1}{c|}{} & 3 & 184.00 & 157.29 & 156.59 & 152.84 & 152.26 & 151.44 & \textbf{151.16} & 137 \\ \hline
\end{tabular}

}}
\end{center}
\label{tab:1Color}
\end{table} \vspace{-0.7cm}

\begin{table}[htbp]
\caption{\small{Running times for graphs with up to 200 vertices from \cite{benchmark1}, in seconds.}}
\begin{center}
\renewcommand{\tabcolsep}{4.5pt}
\renewcommand{\arraystretch}{1.6}
{\scriptsize{
\begin{tabular}{|l|c|c|c|c|c|c|c|c|c|c|}
\hline
Graph & $n$ & $|E|$ &$\rho$,\% &  $k$ &   $\vartheta_k$ & $\vartheta_k'$ & $\theta^3_k$ & $\theta^2_k$ & $\theta^1_k$ & \begin{tabular}[c]{@{}c@{}}$\theta^1_k$\\ $+$ BQP \end{tabular}   \\ \hline

{brock200\_2$^c$} & 200 & 10024 & 50 & 2 &  71 & 99 & 60 & 1333 & 1267 & 1291 \\
 &  &  &  & 3 &  71 & 104 & 61 & 1088 & 1147 & 1168 \\ \hline
{brock200\_4$^c$} & 200 & 6811 & 34 & 2 &  144 & 193 & 115 & 1290 & 1671 & 1697 \\
 &  &  &  & 3 & 143 & 195 & 117 & 1742 & 1575 & 1636 \\ \hline

\multicolumn{ 1}{|l|}{C125.9$^c$} & \multicolumn{ 1}{c|}{125} & \multicolumn{ 1}{c|}{787} & \multicolumn{ 1}{c|}{10} & 2 &  29 & 51 & 29 & 228 & 295 & 1130 \\
\multicolumn{ 1}{|l|}{} & \multicolumn{ 1}{c|}{} & \multicolumn{ 1}{c|}{} & \multicolumn{ 1}{c|}{} & 3 & 32 & 54 & 25 & 322 & 342 & 1806 \\ \hline

\multicolumn{ 1}{|l|}{keller4$^c$} & \multicolumn{ 1}{c|}{171} & \multicolumn{ 1}{c|}{5100} & \multicolumn{ 1}{c|}{35} & 2 &  59 & 114 & 62 & 963 & 1090 & 1217 \\
\multicolumn{ 1}{|l|}{} & \multicolumn{ 1}{c|}{} & \multicolumn{ 1}{c|}{} & \multicolumn{ 1}{c|}{} & 3 &  68 & 125 & 68 & 975 & 1152 & 1190 \\ \hline

\multicolumn{ 1}{|l|}{c-fat200-2$^c$} & \multicolumn{ 1}{c|}{200} & \multicolumn{ 1}{c|}{16665} & \multicolumn{ 1}{c|}{84} & 2 &  7 & 9 & 5 & 940 & 879 & 3627 \\
\multicolumn{ 1}{|l|}{} & \multicolumn{ 1}{c|}{} & \multicolumn{ 1}{c|}{} & \multicolumn{ 1}{c|}{} & 3 & 7 & 8 & 5 & 726 & 616 & 639 \\ \hline
\multicolumn{ 1}{|l|}{c-fat200-5$^c$} & \multicolumn{ 1}{c|}{200} & \multicolumn{ 1}{c|}{11427} & \multicolumn{ 1}{c|}{57} & 2 &  36 & 61 & 35 & 902 & 1199 & 7649 \\
\multicolumn{ 1}{|l|}{} & \multicolumn{ 1}{c|}{} & \multicolumn{ 1}{c|}{} & \multicolumn{ 1}{c|}{} & 3 &  40 & 52 & 33 & 1018 & 1148 & 6955 \\ \hline
\multicolumn{ 1}{|l|}{gen200\_p0.9\_44$^c$} & \multicolumn{ 1}{c|}{200} & \multicolumn{ 1}{c|}{1990} & \multicolumn{ 1}{c|}{10} & 2 &  330 & 507 & 258 & 2176 & 2976 & 13570 \\
\multicolumn{ 1}{|l|}{} & \multicolumn{ 1}{c|}{} & \multicolumn{ 1}{c|}{} & \multicolumn{ 1}{c|}{} & 3 &  330 & 460 & 258 & 3013 & 3528 & 16242 \\ \hline
\multicolumn{ 1}{|l|}{gen200\_p0.9\_55$^c$} & \multicolumn{ 1}{r|}{200} & \multicolumn{ 1}{c|}{1990} & \multicolumn{ 1}{c|}{10} & 2 &  284 & 500 & 338 & 2542 & 3512 & 14187 \\

\multicolumn{ 1}{|l|}{} & \multicolumn{ 1}{r|}{} & \multicolumn{ 1}{r|}{} & \multicolumn{ 1}{c|}{} & 3 &  311 & 475 & 312 & 3268 & 3532 & 10893 \\ \hline

\multicolumn{ 1}{|l|}{san200\_0.7\_2$^c$} & \multicolumn{ 1}{c|}{200} & \multicolumn{ 1}{c|}{5970} & \multicolumn{ 1}{c|}{30} & 2 &  243 & 415 & 220 & 2540 & 3003 & 12066 \\
\multicolumn{ 1}{|l|}{} & \multicolumn{ 1}{c|}{} & \multicolumn{ 1}{c|}{} & \multicolumn{ 1}{c|}{} & 3 &  199 & 373 & 190 & 2593 & 2988 & 5508 \\ \hline
\multicolumn{ 1}{|l|}{san200\_0.9\_2$^c$} & \multicolumn{ 1}{c|}{200} & \multicolumn{ 1}{c|}{1990} & \multicolumn{ 1}{c|}{10} & 2 &  282 & 478 & 353 & 2954 & 3645 & 13590 \\
\multicolumn{ 1}{|l|}{} & \multicolumn{ 1}{c|}{} & \multicolumn{ 1}{c|}{} & \multicolumn{ 1}{c|}{} & 3 &  316 & 515 & 328 & 3555 & 3899 & 11889 \\ \hline
\end{tabular}
}}
\end{center}
\label{tab:1ColorTimes}
\end{table}

\begin{table}[H]
\caption{\small{{Bounds} for  graphs with up to 200 vertices considered by \citet{JanuschowskiBCP}.}}
\begin{center}
\renewcommand{\tabcolsep}{4.5pt}
\renewcommand{\arraystretch}{1.8}
{\scriptsize{\begin{tabular}{|l|c|c|c|c|c|c|c|c|c|c|c|}
\hline
Graph  & $n$ & $|E|$ & $\rho$,\% & $k$ & \begin{tabular}[c]{@{}c@{}}$\alpha_k$  (J$\&$P) \end{tabular} & $\vartheta_k$ & $\vartheta_k'$ & $\theta^3_k$ & $\theta^2_k$ & $\theta^1_k$ & \begin{tabular}[c]{@{}c@{}}$\theta^1_k$\\ $+$ BQP \end{tabular}   \\ \hline
1-FullIns\_4$^{02}$ & 93 & 593 & 14 & 3 & 87 & 93.00 & 93.00 & 92.59 & 92.57 & 92.43 & \textbf{91.33} \\ \hline

4-FullIns\_3$^{02}$ & 114 & 541 & 8 & 3 & 106 & 114.00 & 114.00 & 107.40 & 107.31 & 107.30 & \textbf{107.25} \\ \hline

5-FullIns\_3$^{02}$ & 154 & 792 & 7 & 3 & 144 & 154.00 & 154.00 & 145.33 & 145.25 & 145.25 & \textbf{145.23} \\ \hline

1-Insertions\_4$^{02}$ & 67 & 232 & 10 & 3 & 63 & \textbf{67.00} & \textbf{67.00} & \textbf{67.00} & \textbf{67.00} & \textbf{67.00} & \textbf{67.00} \\ \hline

\multicolumn{ 1}{|l|}{{c-fat200-1$^c$}} & \multicolumn{ 1}{c|}{200} & \multicolumn{ 1}{c|}{18366} & \multicolumn{ 1}{c|}{92} & 6 & 72 & \textbf{72.00} & \textbf{72.00} & \textbf{72.00} & \textbf{72.00} & \textbf{72.00} & \textbf{72.00} \\
\multicolumn{ 1}{|l|}{} & \multicolumn{ 1}{c|}{} & \multicolumn{ 1}{c|}{} & \multicolumn{ 1}{c|}{} & 7 & 84 & \textbf{84.00} & \textbf{84.00} & \textbf{84.00} & \textbf{84.00} & \textbf{84.00} & \textbf{84.00} \\ \hline

c-fat200-1 & 200 & 1534 & 8 & 10 & 180 & 184.67 & 184.67 & 184.67 & 184.67 & \textbf{184.65} & \textbf{184.65} \\ \hline

\multicolumn{ 1}{|l|}{{c-fat200-2$^c$}} & \multicolumn{ 1}{c|}{200} & \multicolumn{ 1}{c|}{16665} & \multicolumn{ 1}{c|}{84} & 7 & 156 & 156.33 & 156.33 & \textbf{156.00} & \textbf{156.00} & \textbf{156.00} & \textbf{156.00} \\
\multicolumn{ 1}{|l|}{} & \multicolumn{ 1}{c|}{} & \multicolumn{ 1}{c|}{} & \multicolumn{ 1}{c|}{} & 8 & 178 & 178.33 & 178.33 & \textbf{178.00} & \textbf{178.00} & \textbf{178.00} & \textbf{178.00} \\ \hline

\multicolumn{ 1}{|l|}{{DSJC125.9$^{02}$}} & \multicolumn{ 1}{c|}{125} & \multicolumn{ 1}{c|}{6961} & \multicolumn{ 1}{c|}{90} & 4 & 16 & \textbf{16.00} & \textbf{16.00} & \textbf{16.00} & \textbf{16.00} & \textbf{16.00} & \textbf{16.00} \\
\multicolumn{ 1}{|l|}{} & \multicolumn{ 1}{c|}{} & \multicolumn{ 1}{c|}{} & \multicolumn{ 1}{c|}{} & 5 & 20 & \textbf{20.00} & \textbf{20.00} & \textbf{20.00} & \textbf{20.00} & \textbf{20.00} & \textbf{20.00} \\
\multicolumn{ 1}{|l|}{} & \multicolumn{ 1}{c|}{} & \multicolumn{ 1}{c|}{} & \multicolumn{ 1}{c|}{} & 6 & 23 & 23.95 & 23.95 & \textbf{23.73} & \textbf{23.73} & \textbf{23.73} & \textbf{23.73} \\ \hline

{gen200\_p0.9\_44} & 200 & 17910 & 90 & 4 & 20 & \textbf{20.00} & \textbf{20.00} & \textbf{20.00} & \textbf{20.00} & \textbf{20.00} & \textbf{20.00} \\ \hline

gen200\_p0.9\_55 & 200 & 17910 & 90 & 4 & 17 & 18.22 & 18.20 & \textbf{18.15} & \textbf{18.15} & \textbf{18.15} & \textbf{18.15} \\ \hline

\multicolumn{ 1}{|l|}{myciel5$^{02}$} & \multicolumn{ 1}{c|}{47} & \multicolumn{ 1}{c|}{236} & \multicolumn{ 1}{c|}{22} & 4 & 44 & \textbf{47.00} & \textbf{47.00} & \textbf{47.00} & \textbf{47.00} & \textbf{47.00} & \textbf{47.00} \\

\multicolumn{ 1}{|l|}{} & \multicolumn{ 1}{c|}{} & \multicolumn{ 1}{c|}{} & \multicolumn{ 1}{c|}{} & 5 & 46 & \textbf{47.00} & \textbf{47.00} & \textbf{47.00} & \textbf{47.00} & \textbf{47.00} & \textbf{47.00} \\ \hline

myciel6$^{02}$ & 95 & 755 & 17 & 3 & 83 & 95.00 & 95.00 & 95.00 & 95.00 & 95.00 & \textbf{93.32} \\ \hline

queen6\_6$^{02}$ & 36 & 580 & 92 & 6 & 32 & 35.97 & 35.97 & 35.84 & 35.84 & \textbf{35.81} & \textbf{35.81} \\ \hline

{san200\_0.9\_1} & 200 & 17910 & 90 & 4 & 16 & 16.10 & 16.08 & \textbf{16.07} & \textbf{16.07} & \textbf{16.07} & \textbf{16.07} \\ \hline

san200\_0.9\_2 & 200 & 17910 & 90 & 4 & 16 & 17.23 & 17.21 & 17.21 & 17.21 & 17.21 & \textbf{17.20} \\ \hline

sanr200\_0.9 & 200 & 17863 & 90 & 4 & 16 & \textbf{17.91} & \textbf{17.91} & \textbf{17.91} & \textbf{17.91} & \textbf{17.91} & \textbf{17.91} \\ \hline

\end{tabular}
}}
\end{center}
\label{tab:2Color}
\end{table}

\begin{table}[H]
\caption{\small{Running times for  graphs with up to 200 vertices from \cite{JanuschowskiBCP}, in seconds.}}
\begin{center}
\renewcommand{\tabcolsep}{4.5pt}
\renewcommand{\arraystretch}{1.4}
{\scriptsize{\begin{tabular}{|l|c|c|c|c|c|c|c|c|c|c|}
\hline
Graph  & $n$ & $|E|$ &$\rho$,\% & $k$ & $\vartheta_k$ & $\vartheta_k'$ & $\theta^3_k$ & $\theta^2_k$ & $\theta^1_k$ & \begin{tabular}[c]{@{}c@{}}$\theta^1_k$\\ $+$ BQP \end{tabular}   \\ \hline
1-FullIns\_4$^{02}$ & 93 & 593 & 14 & 3 & 6 & 8 & 11 & 107 & 103 & 554 \\ \hline

4-FullIns\_3$^{02}$ & 114 & 541 & 8 & 3 & 18 & 22 & 20 & 230 & 196 & 901 \\ \hline

5-FullIns\_3$^{02}$ & 154 & 792 & 7 & 3 & 85 & 100 & 89 & 1,047 & 713 & 3,551 \\ \hline

1-Insertions\_4$^{02}$ & 67 & 232 & 10 & 3 & 2 & 2 & 1 & 14 & 10 & 34 \\ \hline

\multicolumn{ 1}{|l|}{c-fat200-1$^c$} & \multicolumn{ 1}{c|}{200} & \multicolumn{ 1}{c|}{18366} & \multicolumn{ 1}{c|}{92} & 6 & 2 & 2 & 1 & 577 & 776 & 806 \\
\multicolumn{ 1}{|l|}{} & \multicolumn{ 1}{c|}{} & \multicolumn{ 1}{c|}{} & \multicolumn{ 1}{c|}{} & 7 & 2 & 3 & 1 & 739 & 778 & 804 \\ \hline

c-fat200-1 & 200 & 1534 & 8 & 10 & 300 & 742 & 300 & 3,612 & 5,184 & 5,485 \\ \hline

\multicolumn{ 1}{|l|}{c-fat200-2$^c$} & \multicolumn{ 1}{c|}{200} & \multicolumn{ 1}{c|}{16665} & \multicolumn{ 1}{c|}{84} & 7 & 7 & 10 & 5 & 615 & 620 & 638 \\
\multicolumn{ 1}{|l|}{} & \multicolumn{ 1}{c|}{} & \multicolumn{ 1}{c|}{} & \multicolumn{ 1}{c|}{} & 8 & 6 & 10 & 4 & 537 & 644 & 673 \\ \hline

\multicolumn{ 1}{|l|}{DSJC125.9$^{02}$} & \multicolumn{ 1}{c|}{125} & \multicolumn{ 1}{c|}{6961} & \multicolumn{ 1}{c|}{90} & 4 & \textless 1 & \textless 1 & \textless 1 & 103 & 104 & 115 \\

\multicolumn{ 1}{|l|}{} & \multicolumn{ 1}{c|}{} & \multicolumn{ 1}{c|}{} & \multicolumn{ 1}{c|}{} & 5 & \textless 1 & \textless 1 & \textless 1 & 92 & 106 & 118 \\
\multicolumn{ 1}{|l|}{} & \multicolumn{ 1}{c|}{} & \multicolumn{ 1}{c|}{} & \multicolumn{ 1}{c|}{} & 6 & \textless 1 & 1 & \textless 1 & 145 & 141 & 147 \\ \hline

gen200\_p0.9\_44 & 200 & 17910 & 90 & 4 & 3 & 3 & 2 & 955 & 847 & 1,002 \\ \hline

gen200\_p0.9\_55 & 200 & 17910 & 90 & 4 & 5 & 6 & 4 & 1,238 & 1,365 & 1,386 \\ \hline
\multicolumn{ 1}{|l|}{myciel5$^{02}$} & \multicolumn{ 1}{c|}{47} & \multicolumn{ 1}{c|}{236} & \multicolumn{ 1}{c|}{22} & 4 & \textless 1 & \textless 1 & \textless 1 & 1 & 1 & 4 \\

\multicolumn{ 1}{|l|}{} & \multicolumn{ 1}{c|}{} & \multicolumn{ 1}{c|}{} & \multicolumn{ 1}{c|}{} & 5 & \textless 1 & \textless 1 & \textless 1 & 2 & 2 & 7 \\ \hline

myciel6$^{02}$ & 95 & 755 & 17 & 3 & 5 & 6 & 4 & 43 & 37 & 366 \\ \hline

queen6\_6$^{02}$ & 36 & 580 & 92 & 6 & \textless 1 & \textless 1 & \textless 1 & \textless 1 & \textless 1 & 3 \\ \hline

san200\_0.9\_1 & 200 & 17910 & 90 & 4 & 4 & 4 & 3 & 1,103 & 1,282 & 2,897 \\ \hline

san200\_0.9\_2 & 200 & 17910 & 90 & 4 & 4 & 5 & 3 & 963 & 1,136 & 2,036 \\ \hline

sanr200\_0.9 & 200 & 17863 & 90 & 4 & 4 & 4 & 2 & 871 & 888 & 918 \\ \hline

\end{tabular}
}}
\end{center}
\label{tab:2ColorTimes}
\end{table}

Computing  $\theta_k^2(G)$ and $\theta_k^1(G)$ for graphs with more than $200$ vertices is computationally very demanding.
This is due to the presence of two   SDP constraints and a large number of non-negativity constraints in each of the related SDP relaxations.
 Therefore we do not compute $\theta_k^2(G)$ and $\theta_k^1(G)$ for graphs with more than $200$ vertices.
  In Table  \ref{tab:3Color} and \ref{tab:3ColorTimes}  (resp., Table \ref{tab:4Color} and \ref{tab:4ColorTimes}) we present numerical results for graphs  with more than $200$ vertices and compare them with results from  \cite{benchmark1} (resp., \cite{JanuschowskiBCP}).
  Table  \ref{tab:3Color} and  Table \ref{tab:4Color} are organized similarly to Table~\ref{tab:1Color} and Table \ref{tab:2Color}, respectively.
  If a graph with more than $200$ vertices from \cite{benchmark1} and \cite{JanuschowskiBCP} is not present in our computations,
   that means we were not able to compute the corresponding bounds due to a memory issue related to the solver.
   We exclude computations for the Hamming graphs here.

Table   \ref{tab:3Color}  shows that our upper bounds dominate bounds from \cite{benchmark1} also for large graphs.
Our lower bounds also significantly improve upon the lower bounds from the same paper.
Table \ref{tab:4Color}  shows that our SDP bounds are tight for most of the instances.
In Table \ref{tab:4Color} we omitted results for DSJR500.1$^c$ instance since there is discrepancy between our results and the results from \cite{JanuschowskiBCP}.
Namely, for DSJR500.1$^c$  we get $\theta^3_3=36$ and also 36 for the lower bound, while \cite{JanuschowskiBCP} reports $\alpha_3=37$.
Similarly, we get for the same graph that $\theta^3_4=47.24$ (resp., $\theta^3_5=58.24$)  for the lower bound 47 (resp., 58), while \cite{JanuschowskiBCP} reports $\alpha_4=48$ (resp., 59).
Note that computational times in Tables \ref{tab:3ColorTimes} and \ref{tab:4ColorTimes}
indicate that we are able to compute strong bounds for large dense graphs in reasonable time  even without exploiting graph symmetry.

\definecolor{semigray}{gray}{0.9}
\begin{table}[H]
\caption{\small{{Bounds} for  graphs with more than 200 vertices considered by \citet{benchmark1}.}}
\begin{center}
\renewcommand{\tabcolsep}{4.5pt}
\renewcommand{\arraystretch}{1.6}
{\scriptsize{
\begin{tabular}{|l|c|c|c|c|c|c|c|c|c|}
\hline
Graph & $n$ & $|E|$ &$\rho$,\%  & $k$ &  UB (C$\&$C)) & $\vartheta_k$ & $\vartheta_k'$ & $\theta^3_k$ & \begin{tabular}[c]{@{}c@{}} LB  \end{tabular} \\ \hline

C250.9$^c$ & 250 & 3141 & 10 & 2 & 134.50 & 112.48 & \textbf{111.63} & \textbf{111.63} & 83 \\
 &  &  &  & 3 & 201.60 & 168.72 & 167.45 & \textbf{167.15} & 123 \\ \hline

{c-fat500-2$^c$} & 500 & 115611 & 92 & 2 & \textbf{52.00} & \textbf{52.00}& \textbf{52.00} & \textbf{52.00} & 52 \\
 &  &  &  & 3 & \textbf{78.00} & \textbf{78.00} & \textbf{78.00} & \textbf{78.00} & 78 \\ \hline

p\_hat300-1$^c$ & 300 & 33917 & 76 & 2 & 20.90 & 20.14 & \textbf{20.04} & \textbf{20.04} & 14 \\
 &  &  &  & 3 & 31.00 & 30.20 & \textbf{30.06} & \textbf{30.06} & 21 \\ \hline
p\_hat300-2$^c$ & 300 & 22922 & 51 & 2 & 62.10 & 53.93 & 53.43 & \textbf{52.96} & 43 \\
 &  &  &  & 3 & 91.00 & 80.80 & 80.03 & \textbf{77.28} & 61 \\ \hline
\end{tabular}

}}
\end{center}
\label{tab:3Color}
\end{table}

\begin{table}[H] \vspace{-0.7cm}
\caption{\small{Running times for graphs with more than 200 vertices from \cite{benchmark1}, in seconds.}}
\begin{center}
\renewcommand{\tabcolsep}{4.5pt}
\renewcommand{\arraystretch}{1.6}
{\scriptsize{
\begin{tabular}{|l|c|c|c|c|c|c|c|}
\hline
Graph  & $n$ & $|E|$ &$\rho$,\%  & $k$ &   $\vartheta_k$ & $\vartheta_k'$ & $\theta^3_k$ \\ \hline
C250.9$^c$ & 250 & 3141 & 10 & 2 &  1,717 & 3,245 & 1,952 \\
 &  &  &  & 3 &  1,382 & 3,490 & 1,861 \\ \hline
c-fat500-2$^c$ & 500 & 115611 & 92 & 2 & 70 & 115 & 52 \\
 &  &  &  & 3 &  71 & 116 & 47 \\ \hline

p\_hat300-1$^c$ & 300 & 33917 & 76 & 2 &  133 & 174 & 107 \\
 &  &  &  & 3 & 119 & 183 & 109 \\ \hline
p\_hat300-2$^c$ & 300 & 22922 & 51 & 2 &  629 & 1,098 & 626 \\
 &  &  &  & 3 &  636 & 955 & 718 \\ \hline
\end{tabular}
}}
\end{center}
\label{tab:3ColorTimes}
\end{table}

\definecolor{semigray}{gray}{0.9}
\begin{table}[H]
\caption{\small{{Bounds} for  graphs with more than 200 vertices considered by \citet{JanuschowskiBCP}.}}
\begin{center}
\renewcommand{\tabcolsep}{4.5pt}
\renewcommand{\arraystretch}{1.8}
{\scriptsize
{\begin{tabular}{|l|c|c|c|c|c|c|c|c|c|c|c|}
\hline
Graph& $n$ & $|E|$ & $\rho$,\% & $k$ & \begin{tabular}[c]{@{}c@{}}$\alpha_k$ (J$\&$P) \end{tabular}
& $\vartheta_k$ & $\vartheta_k'$ & $\theta^3_k$   \\ \hline
2-FullIns\_4$^{02}$ & 212 & 1621 & 7 & 3 & 202 & 212.00 & 212.00 & \textbf{207.65} \\ \hline

{c-fat500-1$^c$} & 500 & 120291 & 96 & 4 & 56 & \textbf{56.00} & \textbf{56.00} & \textbf{56.00} \\
 &  &  &  & 5 & 70 & \textbf{70.00} & \textbf{70.00} & \textbf{70.00} \\
 &  &  &  & 6 & 84 & \textbf{84.00} & \textbf{84.00} & \textbf{84.00} \\ \hline

{DSJC250.9$^{02}$} & 250 & 27897 & 90 & 3 & 14 & 14.93 & 14.93 & \textbf{14.86} \\
 &  &  &  & 4 & 18 & 19.81 & 19.80 & \textbf{19.72} \\ \hline

{MANN\_a27} & 378 & 70551 & 99 & 15 & 45 & \textbf{45.00} & \textbf{45.00} & \textbf{45.00} \\
 &  &  &  & 17 & 51 & \textbf{51.00} & \textbf{51.00} & \textbf{51.00} \\
 &  &  &  & 20 & 60 & \textbf{60.00} & \textbf{60.00} & \textbf{60.00} \\
 &  &  &  & 22 & 66 & \textbf{66.00} & \textbf{66.00} & \textbf{66.00} \\
 &  &  &  & 25 & 75 & \textbf{75.00} & \textbf{75.00} & \textbf{75.00} \\ \hline
\end{tabular}
}}
\end{center}
\label{tab:4Color}
\end{table}
\begin{table}[H] \vspace{-0.7cm}
\caption{\small{Running times for  graphs with more than 200 vertices from \cite{JanuschowskiBCP}, in seconds.}}
\begin{center}
\renewcommand{\tabcolsep}{4.5pt}
\renewcommand{\arraystretch}{1.4}
{\scriptsize{\begin{tabular}{|l|c|c|c|c|c|c|c|c|c|c|}
\hline
Graph & $n$ & $|E|$ &$\rho$,\% & $k$ & $\vartheta_k$ & $\vartheta_k'$ & $\theta^3_k$  \\ \hline

2-FullIns\_4 & 212 & 1621 & 7 & 3 & 512 & 576 & 993 \\ \hline

c-fat500-1$^c$ & 500 & 120291 & 96 & 4 & 32 & 39 & 20 \\
 &  &  &  & 5 & 28 & 31 & 18 \\
 &  &  &  & 6 & 23 & 31 & 16 \\ \hline

DSJC250.9 & 250 & 27897 & 90 & 3 & 13 & 17 & 16 \\
 &  &  &  & 4 & 18 & 25 & 15 \\ \hline

MANN\_a27 & 378 & 70551 & 99 & 15 & 5 & 6 & 6 \\
 &  &  &  & 17 & 6 & 6 & 8 \\
 &  &  &  & 20 & 7 & 8 & 8 \\
 &  &  &  & 22 & 7 & 7 & 8 \\
 &  &  &  & 25 & 7 & 7 & 7 \\ \hline
\end{tabular}
}}
\end{center}
\label{tab:4ColorTimes}
\end{table}

\subsection{Bounds for the M$k$CS problem for highly symmetric graphs} \label{subsec:symmgraphsNumeric}

In this section we present results for highly symmetric graphs.
In particular, we consider   Johnson $J(v,d,q)$ and Hamming graphs $H(d, q, j)$. For definition and simplified relaxations see Section \ref{sec:GraphSym}.

We first compute our upper bounds for the Johnson graphs $J(v,d,q)$  for which $\alpha_k(G)$ is known.
In particular, for the  Kneser graph  $K(v,d)=J(v,d,0)$ we have the following lower bound on the size of the M$k$CS problem \cite{Furedi}:
\begin{align}
\alpha_k(K(v,d))\ge \sum_{i=1}^k \binom{v-i}{d-1}. \label{eq:lbK}
\end{align}
The lower bound above is tight for $d=2$ and $v\ge k+3$, see \cite{BresarPabon}.
Therefore we use those parameters  in our first experiment  to  demonstrate the quality of our upper bounds for the Kneser graphs.

Table~\ref{tab:K2}  presents results for $d=2$,  $v=10, 15, 20$, and  increasing $k$.
We keep increasing $k$ as long as our bounds are nontrivial,  i.e., smaller than the number of vertices in the graph.
The table reads similarly to the  tables in the previous section.
For all listed examples  except for $J(15,2,0)$ and $k=7$, bounds $\vartheta_k(G)$, $\vartheta_k'(G)$, $\theta_k^3(G)$, $\theta_k^2(G)$, $\theta_k^1(G)$ are equal  to each other.
However we have $\vartheta_k(J(15,2,0))= \vartheta_k'(J(15,2,0))  = \theta_k^3(J(15,2,0)) = \theta_k^2(J(15,2,0))=98$ and  $\theta_k^1(J(15,2,0)) =97.5$.
Our bounds do not provide optimal values, but the ratio of $\theta_k^1(G)$ to $\alpha_k$  does not exceed $1.3$.
The table also indicates that the quality of the upper bound deteriorates when $k$ increases.
We do not present computational times required to solve the relaxations since  they are small, i.e., only a few seconds.

\begin{table}[H]
	\caption{\small{Results for the Kneser graphs $K(v,2)$.
	}}
	\begin{center}
		\renewcommand{\tabcolsep}{4.5pt}
		\renewcommand{\arraystretch}{1.15}
		{\small
			\begin{tabular}{|l|c|c|c|c|c|c|c|}
				\hline
		Graph & $n$ & $|E|$ & $\rho$,\% & $k$ & $\alpha_k$ & $\theta_k^1$ & $\theta_k^1/\alpha_k$ \\ \hline
		$J$(10,2,0) & 45 & 630 & 64 & 1 & 9 & 9.00 & 1.00 \\
		&  &  &  & 2 & 17 & 18.00 & 1.06 \\
		&  &  &  & 3 & 24 & 27.00 & 1.13 \\
		&  &  &  & 4 & 30 & 36.00 & 1.20 \\ \hline
		$J$(15,2,0) & 105 & 4095 & 75 & 1 & 14 & 14.00 & 1.00 \\
		&  &  &  & 2 & 27 & 28.00 & 1.04 \\
		&  &  &  & 3 & 39 & 42.00 & 1.08 \\
		&  &  &  & 4 & 50 & 56.00 & 1.12 \\
		&  &  &  & 5 & 60 & 70.00 & 1.17 \\
		&  &  &  & 6 & 69 & 84.00 & 1.22 \\
		&  &  &  & 7 & 77 & 97.50 & 1.27 \\ \hline
		$J$(20,2,0) & 190 & 14535 & 81 & 1 & 19 & 19.00 & 1.00 \\
		&  &  &  & 2 & 37 & 38.00 & 1.03 \\
		&  &  &  & 3 & 54 & 57.00 & 1.06 \\
		&  &  &  & 4 & 70 & 76.00 & 1.09 \\
		&  &  &  & 5 & 85 & 95.00 & 1.12 \\
		&  &  &  & 6 & 99 & 114.00 & 1.15 \\
		&  &  &  & 7 & 112 & 133.00 & 1.19 \\
		&  &  &  & 8 & 124 & 152.00 & 1.23 \\
		&  &  &  & 9 & 135 & 171.00 & 1.27 \\ \hline
	\end{tabular}}
\end{center}
\label{tab:K2}
\end{table}

For $k=2$ and  $v \ge \tfrac{1}{2}(3+\sqrt{5})d$, there exists the following upper bound on $\alpha_2(K(v,d))$, see \cite{FurediFrankl}:
\begin{align}
\alpha_2(K(v,d))\le  \binom{v-1}{d-1}+ \binom{v-2}{d-1}. \label{eq:ubK}
\end{align}
Therefore, in our next experiment we first compute  bounds~\eqref{eq:lbK} and \eqref{eq:ubK} for $k=2$ and various $(v,d)$ s.t., $v \ge \tfrac{1}{2}(3+\sqrt{5})d$ in order to
find graphs for which these two bounds coincide.
 Table~\ref{tab:Kother} presents examples of such graphs.
Although for each graph in Table~\ref{tab:Kother} all our upper bounds are equal,
there are existing Kneser graphs for which our relaxations do not provide the same bounds.
For example, for the Kneser graph $J(14,5,0)$ and $k=2$ we have $\vartheta_k(J(14,5,0))=\theta_k^2(J(14,5,0)) = 1430$ and  $\theta_k^1(J(14,5,0)) = 1386$ while $\alpha_k(J(14,5,0))=1210$.

\begin{table}[H]
\caption{\small{Results for the Kneser graphs $K(v,d)$, $d\ge 3$ and $k=2$.
		}}
\begin{center}
	\renewcommand{\tabcolsep}{4.5pt}
	\renewcommand{\arraystretch}{1.5}
	{\small
	\begin{tabular}{|l|c|c|c|c|c|c|}
		\hline
		Graph & $n$ & $|E|$ & $\rho$,\% & $\alpha_k$ & $\theta_k^1$ & $\theta_k^1/\alpha_k$ \\ \hline
		$J$(15,3,0) & 455 & 50050 & 48 & 169 & 182.00 & 1.08 \\ \hline
		$J$(16,3,0) & 560 & 80080 & 51 & 196 & 210.00 & 1.07 \\ \hline
		$J$(17,3,0) & 680 & 123760 & 54 & 225 & 240.00 & 1.07 \\ \hline
		$J$(18,3,0) & 816 & 185640 & 56 & 256 & 272.00 & 1.06 \\ \hline
		$J$(16,4,0) & 1820 & 450450 & 27 & 819 & 910.00 & 1.11 \\ \hline
		$J$(17,4,0) & 2380 & 850850 & 30 & 1015 & 1120.00 & 1.10 \\ \hline
		$J$(18,4,0) & 3060 & 1531530 & 33 & 1240 & 1360.00 & 1.10 \\ \hline
		$J$(16,5,0) & 4368 & 1009008 & 11 & 2366 & 2730.00 & 1.15 \\ \hline
		$J$(17,5,0) & 6188 & 2450448 & 13 & 3185 & 3640.00 & 1.14 \\ \hline
		$J$(18,5,0) & 8568 & 5513508 & 15 & 4200 & 4760.00 & 1.13 \\ \hline
	\end{tabular}}
\end{center}
\label{tab:Kother}
\end{table}
For $k=2$ and the same $v$ and $d$ as those used in Table~\ref{tab:Kother}, we compute upper and lower bounds for $J(v,d,d-1)$ graphs, see Table~\ref{tab:Jother}.
Those graphs are more sparse than the corresponding Kneser graphs.
The table shows that $\alpha_2 (J(15,3,2))=70$, see also Table~\ref{tab:tight}. For each graph in the table all our upper bounds are equal.
Gaps between upper and lower bounds  for graphs  $J(v,d,d-1)$ are larger on average than gaps between  upper  bounds and $\alpha_k(G)$ for $J(v,d,0)$.
Table~\ref{tab:Jother} might be used as a benchmark for highly symmetric graphs.
\begin{table}[H]
\caption{\small{Results for the Johnson graphs $J(v,d,d-1)$, $d\ge 3$ and $k=2$.
 }}
\begin{center}
\renewcommand{\tabcolsep}{4.5pt}
\renewcommand{\arraystretch}{1.8}
{\small
\begin{tabular}{|l|c|c|c|c|c|c|}
\hline
Graph  & $n$ & $|E|$ & $\rho$,\% & $\theta_k^1$ & LB  & $\theta_k^1/LB $ \\ \hline
$J(15,3,2)$ & 455 & 8190 & 8 & 70.00 & 70 & 1.00 \\ \hline
$J(16,3,2)$ & 560 & 10920 & 7 & 80.00 & {73} & 1.09 \\ \hline
$J(17,3,2)$ & 680 & 14280 & 6 & 90.67 & 88 & 1.03 \\ \hline
$J(18,3,2)$ & 816 & 18360 & 6 & 102.00 & {94} & 1.09 \\ \hline
$J(16,4,3)$ & 1820 & 43680 & 3 & 280.00 & {269} & 1.04 \\ \hline
$J(17,4,3)$ & 2380 & 61880 & 2 & 340.00 & 291 & 1.17 \\ \hline
$J(18,4,3)$ & 3060 & 85680 & 2 & 408.00 & 367 & 1.11 \\ \hline
$J(16,5,4)$ & 4368 & 120120 & 1 & 728.00 & 565 & 1.29 \\ \hline
$J(17,5,4)$ & 6188 & 185640 & 1 & 952.00 & 771 & 1.23 \\ \hline
$J(18,5,4)$ & 8568 & 278460 & $<$1 & 1224.00 & 974 & 1.26 \\ \hline
\end{tabular} }
\\\vspace{0.3cm}
\end{center}
\label{tab:Jother}
\end{table}

\subsection{Lower bounds on the chromatic number} \label{subsec:boundChi}

In this section we further exploit upper bounds on the M$k$CS problem for deriving  lower bounds on the chromatic number of a graph $G=(E,V)$.
In particular, if an upper bound on $\alpha_k(G)$ is smaller than the number of vertices in the graph, then $G$ is not $k$-colorable.
Hence the lower bound on $\chi(G)$ is at least $k+1$. Using this principle, we obtain a lower bound on $\chi(G)$ as
\begin{equation}\label{chi}
\Psi(G) =  \max \{ ~k: \mbox{upper bound on } \alpha_k < |V|\} +1.
\end{equation}

We test below  our lower bound  on the chromatic number \eqref{chi} for several graphs.
We first consider vertex-transitive graphs from \citet{DukanovicRendl}, Table 10.
Table \ref{tab:JGenDR} presents results for  two out of four Johnson graphs  and   several Hamming graphs from  \cite{DukanovicRendl}.
`n.a.' in the table means that we could't compute bounds.
For all graphs in Table \ref{tab:JGenDR}, we have that  $\vartheta_k(G)$ differs from $\theta_k^1(G)$.
Table  \ref{tab:JGenDR}  shows that $\Psi(J(12,7,3))=7$ and $\Psi(J(14,7,3))=12$, which corresponds to the lower bounds
on the chromatic number of the corresponding graphs from \cite{DukanovicRendl}.
These lower bounds are obtained from \eqref{chi} with  $\theta_k^1(G)$ as an upper bound.
Moreover, we obtain the same lower bounds on $\chi(G)$  as the authors from \cite{DukanovicRendl}  for all Johnson and Hamming graphs from Table 10, in \cite{DukanovicRendl}.

It is interesting to note that the lower bound  \eqref{chi} with $\vartheta_k(G)$  as an upper bound,
provides the same  lower bound on the chromatic number of $G$  as
$\vartheta( \bar{G} )$, where $\bar{G}$ denotes the graph complement of $G$, for  all graphs from Table 10, in \cite{DukanovicRendl}.
Table  \ref{tab:JGenDR}  also shows that for  $J(12,7,3)$ and $k=1,2,3$ we compute optimal solutions for the M$k$CS problem,
and  a similar conclusion follows  for $H(6,2,4)$ and $H^-(12,2,7)$.

\begin{table}[H]
\caption{\small{Results for the Johnson and Hamming graphs  from \cite{DukanovicRendl}, Table 10.
}}
\begin{center}
\renewcommand{\tabcolsep}{4.5pt}
\renewcommand{\arraystretch}{1.1}
{\small
\begin{tabular}{|l|c|c|c|c|c|c|c|c|c|}
\hline
Graph & $n$ & $|E|$ & $\rho$, \%  & $k$ & $\vartheta_k$ & $\theta_k^1$ &    LB &  $\theta_k^1$/LB \\ \hline
{$J(12,7,3)$} & 792 & 69,300 & 22 & 1 & 214.50 & 120.00 & 120 & 1.00 \\
 &  &  &  & 2 & 429.00& 240.00 & 240 & 1.00 \\
 &  &  &  & 3 & 643.50&  360.00 & 360 & 1.00 \\
 &  &  &  & 4 & 792.00 &480.00 & 437 & 1.09 \\
  &  &  &  & 5 & 792.00 & 600.00 & 522 & 1.15 \\
 &  &  &  & 6 & 792.00 & 720.00 & 584 & 1.23 \\
 &  &  &  & 7 & 792.00 & 792.00 & 640 & 1.23 \\\hline

{$J(14,7,3)$} & 3432 & 2,102,100 & 36 & 7 & 3003.00 & 2032.80 & 1262 & 1.61 \\
 &  &  &  & 8 & 3432.00 & 2323.20 & 1428 & 1.63 \\
 &  &  &  & 11 & 3432.00 & 3194.40 & n.a. & n.a.\\
  &  &  &  & 12 &  3432.00 & 3432.00 & n.a. & n.a. \\\hline

  H(6,2,4)  & 64 &  480  & 23 & 1 & 16.00 &12.00 & 12 & 1.00 \\
 &  &  &  & 2 &32.00  & 24.00 & 24 & 1.00 \\
 &  &  &  & 3 & 48.00 & 36.00 & 36 & 1.00 \\
 &  &  &  & 4 & 64.00 &48.00 & 48 & 1.00 \\
& & &  & 5 & 64.00 & 60.00 & 52 & 1.15 \\
 &  &  &  & 6 & 64.00 & 64.00 & 60 & 1.07 \\ \hline

 H(10,2,8) & 1024 & 23,040 & 4 & 2 & 768.00 & 640.00 & 520 & 1.23 \\
 & & & & 3 & 1024.00 & 960.00 & 700 & 1.37 \\
 &  &  &  & 4 &1024.00  & 1024.00 & 880 & 1.16 \\\hline

 $H^-(12,2,7)$ & 4096 &  6,760,448 & 81 & 1 &  7.70 & 4.00 & 4 & 1.00 \\
 &  &  &  & 2 & 15.41 & 8.00 & 8 & 1.00 \\
 &  &  &  & 3 & 23.11 & 12.00 & 12 & 1.00 \\
 &  &  &  & 4 & 30.81 & 16.00 & 16 & 1.00 \\
 &  & &  & 5 &  38.52 & 20.00 & 20 & 1.00 \\
 &  & &  & 531 &  4090.70 & 2124.00  & n.a. & n.a.   \\
 &  & &  & 532 &  4096.00 & 2128.00  & n.a.& n.a.   \\
 &  & &  & 1023 &  4096.00 & 4092.00 & n.a. & n.a.\\
 &  & &  & 1024 &  4096.00 &  4096.00 &n.a. & n.a.\\\hline
\end{tabular} }
\end{center}
\label{tab:JGenDR}
\end{table}

 Further, we consider graphs from the COLOR02 symposium~\cite{color02} whose lower bounds on $\chi(G)$ were also computed in \cite{DukanovicRendl}.
 We compute bounds on the chromatic number for all mentioned graphs except for $4$-Insertions-$4$ that has 475 vertices.
 For all graphs except myciel7, we obtain the same lower bounds on the chromatic number of a graph as those reported in  \cite{DukanovicRendl}.
Moreover, for many graphs it is sufficient to compute $\theta^3_k(G)$ to obtain the best lower bound.

For the myciel7 graph we  improve  the lower bound on the chromatic number.
Namely,  it follows from Table \ref{tab:newChi1} that our lower bound on  $\chi({\rm myciel7})$ is four while the lower bound  from \cite{DukanovicRendl} is three.
It is also known that the chromatic number for myciel7 is eight.

\begin{table}[H]
\caption{\small{Results for myciel7 graph. }}\label{tab:newChi1}
\begin{center}
\renewcommand{\tabcolsep}{4.5pt}
\renewcommand{\arraystretch}{1.1}
{\small
\begin{tabular}{|l|c|c|c|c|c|c|c|}
\hline
Graph & $n$ & $|E|$ & $\rho$, \%  & $k$ &  $\theta_k^1$  & $\theta_k^1$ with BQP  \\ \hline
myciel7  & 191 & 2360 & 13 & 3 & 191 & 186.84 \\
         &     &      &    & 4 & 191 & 191 \\\hline
\end{tabular} }
\end{center}
\end{table}

Notice that we are able to improve the lower bound for $\chi({\rm myciel7})$ from \cite{DukanovicRendl} only when using  our strongest relaxation that includes  BQP inequalities.
However, for other graphs from the COLOR02 symposium~\cite{color02} even that relaxation does not help to improve the results.
Finally, we also compute $\theta_k^1(G)$ with  BQP inequalities for $H(6,2,4)$ from Table \ref{tab:JGenDR}, but the lower bound on  $\chi(H(6,2,4))$ remains six.

The results in this section show that one can compute strong lower bounds    on the chromatic number of a graph by exploiting upper bounds for the M$k$CS problem.

\section{Conclusion} \label{sec:conclusion}

This paper combines several modelling approaches to derive strong bounds  for the maximum $k$-colorable subgraph problem and related problems.

We first  analyze the existing upper bound for  the M$k$CS  problem known as the generalized $\vartheta$-number,  see \eqref{pr:theta4}.
Then, we strengthen it by adding non-negativity constraints to the corresponding SDP relaxation.
We call the resulting upper bound the generalized $\vartheta'$-number, see \eqref{pr:theta'}.
Then, we propose several new SDP relaxations for the M$k$CS  problem with increasing complexities.
The  sizes of our new SDP relaxations initially depend on the number of colors $k$ and the number of vertices in the graph,  see \eqref{pr:sdp1}, \eqref{pr:sdp2} and \eqref{pr:sdp3}.
 To reduce the sizes of those three SDP relaxations, we exploit the fact that the  M$k$CS  problem is invariant  with respect to  color permutations.
 The reduction results in the SDP relaxations with at most two SDP constraints of order at most $(n{+}1)$ for any $k$ and any graph type, see Theorem \ref{thm:SymColor2},
 Corollary \ref{cor:SymColor3} and Theorem \ref{thm:SymColor1}.
The resulting relaxations provide the following upper bounds for the M$k$CS problem: $\theta^1_k$ see \eqref{pr:sdpSymColor2},
$\theta^2_k$ see  \eqref{pr:sdpSymColor2} without constraints \eqref{ineq:symColor3},~\eqref{ineq:symColor4}, and $\theta^3_k$ see \eqref{pr:sdpSymColor1}.
In Proposition \ref{prop:tight} we characterise a family of graphs for which those bounds are tight.
To  improve our strongest relaxation we add non-redundant, symmetry-reduced, boolean quadric polytope inequalities, see \eqref{ineq:31}--\eqref{ineq:32}.

We further reduce relaxations for several classes of highly symmetric graphs including the Johnson and Hamming graphs, see Section \ref{sec:GraphSym}.
The resulting relaxations are linear programs or linear programs with one convex quadratic constraint.
Finally, we show that the vector and matrix lifting relaxations for  the max-$k$-cut problem and the $k$-equipartition problem
are equivalent by exploiting the invariance of the problems under permutations of the subsets, see Section \ref{subsec:partit}.

We compute  upper and lower bounds for graphs  considered in~\cite{benchmark1,JanuschowskiBCP} with up to $500$ vertices.
We also compute bounds for the M$k$CS problem for highly symmetric graphs with up to $6,760,448$ edges.
We solve the problem for several graphs to optimality and  obtain stronger bounds than in \cite{benchmark1} for all but one tested graphs.
 Our lower bounds on the chromatic number of a graph are competitive  with bounds from the literature.

\medskip\medskip\medskip

\noindent
{\bf Acknowledgment.}
The authors are grateful to Monique Laurent for raising a question about the relation between the maximum $k$-colorable subgraph   and   the maximum stable set problem.
The second author would also like to thank  Renate  van der Knaap for implementing a tabu search algorithm and  for computing lower bounds.
We would also like to thank two anonymous referees for suggestions that led to an improvement of this paper.

\bibliographystyle{plainnat}
\bibliography{k_subgraph}

\end{document}